\documentclass[11pt]{article}

\usepackage[a4paper, left=3cm, top=3cm, right=3cm, bottom=3cm ]{geometry}
\usepackage{authblk}
\usepackage{appendix}

\usepackage{hyperref}
\usepackage{amsmath,amssymb}
\usepackage{amsthm}
\usepackage{mathtools}
\usepackage{enumitem}

\usepackage{tikz}
\usetikzlibrary{decorations.pathmorphing}\usepackage{graphicx}
\usepackage{xcolor}

\newtheorem{theorem}{Theorem}[section]
\newtheorem{lemma}[theorem]{Lemma}
\newtheorem{proposition}[theorem]{Proposition}

\newtheorem{assumption}{Assumption}

\theoremstyle{definition}
\newtheorem{definition}[theorem]{Definition}
\newtheorem{remark}[theorem]{Remark}
\newtheorem{example}[theorem]{Example}

\newcommand{\R}{\mathbb{R}}
\newcommand{\N}{\mathbb{N}}

\newcommand{\X}{\mathbb{X}}
\newcommand{\Y}{\mathbb{Y}}
\newcommand{\E}{\mathbb{E}}

\newcommand{\Ao}{\mathbf{A}}

\newcommand{\To}{\mathbf{T}}
\newcommand{\Mo}{\mathbf{M}}
\newcommand{\Bo}{\mathbf{B}}
\newcommand{\Do}{\mathbf{D}}
\newcommand{\doo}{\mathbf{d}}

\newcommand{\Rc}{\mathcal{R}}
\newcommand{\Lc}{\mathcal{L}}
\newcommand{\Dc}{\mathcal{D}}

\newcommand{\Tc}{\mathcal{T}}

\newcommand{\la}{\lambda}
\newcommand{\ka}{\kappa}
\newcommand{\al}{\alpha}
\newcommand{\ph}{\varphi}
\newcommand{\ga}{\gamma}

\newcommand{\La}{\Lambda}
\newcommand{\Ga}{\Gamma}

\newcommand{\uo}{\mathbf u}
\newcommand{\vo}{\mathbf v}

\newcommand{\kao}{{\boldsymbol \kappa}}

\newcommand{\qone}{s}

\setlist[enumerate]{label = (\alph*),topsep=0em,parsep=0em}
\setlist[itemize]{topsep=0em,parsep=0em}

\newcommand\norm[1]{{\left\Vert#1\right\Vert}}
\newcommand\inner[2]{\left\langle#1,#2\right\rangle}
\newcommand\abs[1]{\left\vert#1\right\vert}

\newcommand\snorm[1]{\Vert#1\Vert}
\newcommand\sinner[2]{\langle#1,#2\rangle}
\newcommand\sabs[1]{{\lvert#1\rvert}}

\newcommand{\sign}{\operatorname{sign}}
\newcommand{\dom}{\operatorname{dom}}
\newcommand{\ran}{\operatorname{ran}}

\DeclareMathOperator*{\argmin}{argmin}

\newcommand{\prox}{\operatorname{prox}}
\newcommand{\id}{\operatorname{id}}

\newcommand{\fix}{\operatorname{Fix}}
\newcommand{\Lip}{\operatorname{Lip}}

\newcommand{\mpi}{+}

\numberwithin{equation}{section}
\numberwithin{theorem}{section}

\allowdisplaybreaks

\usepackage{soul}
\colorlet{lred}{red!40}
\colorlet{lgreen}{green!40}
\colorlet{lblue}{blue!40}

\title{Convergence of non-linear diagonal frame filtering for regularizing inverse problems}

\date{Initial submission August 30, 2023\\ This version: February 24, 2024}

\author{Andrea Ebner}

\affil{Department of Mathematics, University of Innsbruck\authorcr
Technikerstrasse 13, 6020 Innsbruck, Austria\authorcr
E-mail:  \texttt{andrea.ebner@uibk.ac.at}
 }

\author{Markus Haltmeier}

\affil{Department of Mathematics, University of Innsbruck\authorcr
Technikerstrasse 13, 6020 Innsbruck, Austria\authorcr
E-mail:  \texttt{markus.haltmeier@uibk.ac.at}
 }

\begin{document}

\maketitle

\begin{abstract}
Inverse problems are key issues in several scientific areas, including signal processing and medical imaging. Since inverse problems typically suffer from instability with respect to data perturbations, a variety of regularization techniques have been proposed. In particular, the use of filtered diagonal frame decompositions has proven to be effective and computationally efficient. However, existing convergence analysis applies only to linear filters and a few non-linear filters such as soft thresholding. In this paper, we analyze filtered diagonal frame decompositions with general non-linear filters. In particular, our results generalize SVD-based spectral filtering from linear to non-linear filters as a special case. As a first approach, we establish a connection between non-linear diagonal frame filtering and variational regularization, allowing us to use results from variational regularization to derive the convergence of non-linear spectral filtering. In the second approach, as our main theoretical results, we relax the assumptions involved in the variational case while still deriving convergence. Furthermore, we discuss connections between non-linear filtering and plug-and-play regularization and explore potential benefits of this relationship.

\end{abstract}

\section{Introduction}

Let $\Ao \colon \X \to \Y$ be a bounded linear operator between two real Hilbert spaces $\X$ and $\Y$. We consider the inverse problem of recovering  $x^\mpi \in \X$ from noisy data
\begin{equation} \label{eq:ip}
y^\delta  = \Ao x^\mpi + z \,,
\end{equation}
where $z$ is the data perturbation with $\snorm{z} \leq \delta$ for some noise level $\delta > 0$. Inverting the operator $\Ao$ is often ill-posed in the sense that the the Moore-Penrose inverse $\Ao^\mpi$ is discontinuous. Thus, small errors in the data are significantly amplified by use of exact solution methods. To address this problem, regularization methods have been developed with the aim of finding approximate but stable solution strategies \cite{Be17,En96,rieder2013keine}.

\subsection{Diagonal frame filtering}

Diagonal frame decompositions in combination with regularizing filters are a flexible and efficient regularization concept  for \eqref{eq:ip}. Suppose $\Ao$ has a diagonal frame decomposition (DFD) giving the representations
\begin{align*}
\Ao &=
\sum_{\la \in \La} \ka_\la \inner{\,\cdot\,}{u_\la} \bar{v}_\la
\\
\Ao^\mpi &= \sum_{\la \in \La} \ka_\la^{-1} \sinner{\,\cdot\,}{v_\la}  \bar{u}_\la
\,.
\end{align*}
Here $(u_\la)_{\la \in \La}$ and $(v_\la)_{\la \in \La}$ are frames of  $\ker(\Ao)^\bot$ and $\overline{\ran(\Ao)}$, respectively,  with corresponding  dual frames $(\bar{u}_\la)_{\la \in \La}$ and $(\bar{v}_\la)_{\la \in \La}$.
In the special case that $(u_\la)_\la$ and  $(v_\la)_\la$ are orthonormal, $\Ao = \sum_{\la \in \La} \ka_\la \inner{\cdot}{u_\la} v_\la$ is a singular value decomposition (SVD) for  $\Ao$. More general frame decompositions have first been  studied by Candés and Donoho \cite{Ca02, donoho1995nonlinear} in the context of statistical estimation and recently in \cite{frikel2018efficient,Fr19,ebner2023regularization,goppel2023translation,quellmalz2023frame,hubmer2021frame,hubmer2022regularization} in  the context of regularization theory.   Specifically, in  \cite[Theorem 2.10]{ebner2023regularization} it has been shown that if $(\ka_\la)_{\la \in \La}$ accumulate at zero, then $\Ao^\mpi$ is unbounded.  In such a situation, regularization  methods have to applied for the solution of \eqref{eq:ip}.

Linear filtered DFD methods  are based  on a regularizing filter family $(f_\al)_{\al >0}$   and  defined by
$\Bo_{\al}(y^\delta) \coloneqq \sum_{\la \in \La} f_\al(\ka_\la) \sinner{y^\delta}{v_\la} \bar{u}_\la
$; see \cite{ebner2023regularization}. Main theoretical questions are the stability of $\Bo_{\al}$ and the convergence for $\delta, \al \to 0$. Each factor $f_\al(\ka_\la) \sinner{y^\delta}{v_\la}$ is a damped version of the exact coefficient inverse $\ka_\la^{-1} \sinner{y^\delta}{v_\la}$. The  filtering process $\ph_\al  ( \ka_\la, \sinner{y^\delta}{v_\la}) =  (f_\al(\ka_\la) \ka_\la) \cdot  \sinner{y^\delta}{v_\la}$ is  linear in $\sinner{y^\delta}{v_\la}$ and represented  by the damping factors $f_\al(\ka_\la) \ka_\la$.  Further note  that $\Bo_{\al}$ reduces to the well-known spectral filtering technique which  in the regularization context is also referred  to as filter-based regularization. The convergence analysis in this special case is  well known and can be found for example in \cite{En96,groetsch1984theory}. The more general case of filtered  DFDs has been first analyzed  in \cite{ebner2023regularization} and later in \cite{hubmer2022regularization}.

\subsection{Non-linear extension}

A major drawback of linear regularizing filters is that the damping factor depends only on the quasi-singular values and is independent of the data. In practice, certain filters that depend non-linearly on $\sinner{y^\delta}{v_\la}$ tend to perform better in filtering out noise than linear methods; see \cite{An21,Ga98,mallat1999wavelet}. The aim  of this paper is to analyze general non-linear frame-based diagonal filtering
\begin{equation} \label{eq:non-linear}
\Bo_\al (y^\delta)
 = \sum_{\la \in \La} \ka_\la^{-1} \ph_\al(\ka_\la,\sinner{y^\delta}{v_\la}) \bar u_\la  \,,
\end{equation}
where  $(\ph_\al)_{\al > 0}$ is  a non-linear  filter rigorously introduced in Definition~\ref{def:filter} below.
The reconstruction mappings $\Bo_\al \colon \Y \to \X$ come with a clear interpretation:  In order to avoid  noise amplification due to  multiplication with  $\ka_\la^{-1}$, the filter damps each coefficient according to damping factors $\ph_\al(\ka_\la,\sinner{y^\delta}{v_\la})$ prior to the inversion.
By taking $\ph_\al  ( \ka_\la, \sinner{y^\delta}{v_\la}) =  (f_\al(\ka_\la) \ka_\la) \cdot  \sinner{y^\delta}{v_\la} $ we recover linear filtering analyzed in \cite{ebner2023regularization}.   An example of a non-linear filter is the soft thresholding filter defined by
$\ph_{\al}(\ka,c) =  \sign{(c)}(\abs{c} - \al/\ka)_+$ where $(\cdot)_+ \coloneqq \max \{ \cdot, 0\} $.
In  \cite{Fr19} it has been shown that this filter  yields  a convergent regularization method. Affine filters applied with the SVD have recently been studied in \cite{arndt2023invertible}. Our paper extends these special cases to general classes of non-linear filters.

Specifically, we establish assumptions on the non-linear regularizing filter $(\ph_{\al})_{\al > 0}$ to demonstrate stability and convergence  in the following sense :
\begin{itemize}
\item \textit{Stability: For fixed $\al > 0$ let $y^\delta, y^k \in \Y$ be with  $y^k \to y^\delta$}. Then we have $\Bo_\al (y^k) \rightharpoonup \Bo_\al (y^\delta)$ as $k \to \infty$.
\item \textit{Convergence:} Consider positive null sequences $(\delta_k)_k, (\al_k)_k$ with $\delta_k^2/\al_k \to 0$. Let  $y \in \ran(\Ao)$ and  $(y^k)_k \in \Y^\N$ with  $\snorm{y^k -y} \leq \delta_k$.
Then $\Bo_{\al_k} (y^k) \rightharpoonup \Ao^\mpi y$ as $k \to \infty$.
\end{itemize}

As our main theoretical findings, we derive such results for a broader class of non-linear filters by exploiting the specific diagonal structure of \eqref{eq:non-linear}; see Section~\ref{ssec:case2}. In addition, for a broad class of homogeneous filters, such results will be derived by a reduction to the well-established case of variational regularization \cite{scherzer2009variational}; see Section~\ref{ssec:case1}. Note that even in the case where we reduce our analysis to variational regularization, the non-linear filtered regularization technique is related but different from variational regularization with separable constraints \cite{grasmair2011linear,bredies2009regularization}. This is discussed in detail for  special case of soft thresholding in \cite{Fr19}, where the  diagonal frame filtering has been opposed to frame-analysis and frame-synthesis regularization.

Note that just as we can express the filtered DFD as a variational regularization, we can also express it as a plug-and-play (PnP) regularization. While convergence results in the context of regularization methods for PnP have not been extensively studied, with the first in-depth study presented in \cite{ebner2022plug}, it is important to acknowledge that the assumptions for convergence of PnP are particularly stringent. A specific class of nonlinear filters conform to the PnP regularization by satisfying the necessary conditions, and as a result the established stability and convergence results hold, even showing stronger stability in these cases. Conversely  because the assumptions of the analysis of the filtered DFD are quite general, it covers various PnP methods that aren't addressed in the theory presented in \cite{ebner2022plug}. Consequently, as discussed in Section~\ref{sec:pnp}, this paper even contributes the regularization theory of PnP.

\subsection{Outline}

In Section~\ref{sec:prelim} we present preliminaries in the form of a notion, auxiliary results, and some technical lemmas. In Section~\ref{sec:non-linear} we rigorously introduce non-linear filters and non-linear filtered diagonal frame decompositions. The main results of this paper are presented in Section~\ref{sec:conv}, where we provide two approaches to proving convergence: by linking to existing theories of variational regularization for stationary filters, and by a direct proof in the general case.Further, in Section~\ref{sec:pnp} we analyze connections between filtered DFD and PnP regularization. In Section~\ref{sec:conclusion}, we summarize our findings and offer some future research directions.

\section{Preliminaries}\label{sec:prelim}

Let $\X$, $\Y$ be Hilbert spaces. If $\Bo \colon \dom(\Bo) \subseteq \X \rightarrow \Y$, then  $\dom(\Bo)$ denotes the domain and $\ran(\Bo) = \Bo(\dom(\Bo))$ the range of $\Bo$.   If $\Bo$  is linear bounded with $\dom(\Bo) =\X$ we write $\Bo \colon \dom(\Bo^\mpi) \subseteq \Y  \to \X$ with $\dom(\Bo^\mpi) \coloneqq \ran(\Bo) \oplus \ran(\Bo)^\bot$ for the Moore-Penrose inverse of $\Bo$.
Recall that   $\Phi \colon \X \to \X$ is nonexpansive if  $\forall x, y \in \X \colon \norm{\Phi(x)-\Phi(y)} \leq \norm{x-y} $.

\subsection{Functionals and proximity operators}

Functionals on $\X$ will be written as $\Rc \colon \X \rightarrow \R \cup \{\infty\}$ and we  usually use $r$ or $\qone$ to denote a functional when $\X = \R$. We define the domain of  $\Rc$ by $\dom(\Rc) \coloneqq \{x \in \X \mid \Rc(x) < \infty \}$ and for $q \in \R$ we define the lower level set of $\Rc$ with bound $q$ by $ \Lc(\Rc, q) \coloneqq \{x \in \X \mid \Rc(x) \leq q \}$.
The functional $\Rc$ is called proper if $\dom(\Rc) \neq \emptyset$ and convex if $\Rc(t x + (1-t)y) \leq t \Rc(x) + (1-t) \Rc(y)$  for all $x, y \in \X$ and $t \in [0,1]$.  It is called lower semi-continuous if  $\Lc(\Rc,q)$ is  closed
for all $q \in \R$.  For convex functions sequentially, weak as well as weak sequential lower semi-continuity are equivalent to strong lower semi-continuity. We call $\Rc$ norm-coercive if $\Rc(x^k) \to \infty$ for all  $(x^k)_{k\in \N} \in \dom(\Rc)^\N$ with $\snorm{x^k} \to \infty$.

We define $\Ga_0(\X)$ as the set of all $\Rc\colon \X \rightarrow \R \cup \{\infty\}$ that are proper, convex and lower semi-continuous.
The subdifferential of $\Rc \in \Ga_0(\X)$ is a set-valued operator $\partial \Rc \colon \X \rightarrow 2^\X$ defined by
\begin{equation*}
\partial \Rc ( x) \coloneqq  \{u \in \X \mid \forall y \in \X \colon \inner{y-x}{u} + \Rc(x) \leq \Rc(y) \} \,.
\end{equation*}
The elements of $\partial \Rc(x)$ are called subgradients of $\Rc$ at $x$.

Let $\La$ be a countable index set and  $r_\la \colon \R \to \R \cup \{ \infty \}$ be non-negative for   $\la \in \La$.  The functional $\Rc = \bigoplus_{\la \in \La} r_\la \colon \ell^2(\La) \to \R \cup \{\infty\}$  is defined by $\Rc((x_\la)_\la ) = \sum_{\la \in \La} r_\la(x_\la) $.

For $\Rc \in \Ga_0(\X)$, the proximity operator  $\prox_{\Rc} \colon \X \to \X $ is well defined by $\prox_{\Rc}(x) \coloneqq  \argmin_{y \in \X}  \norm{x-y}^2/2 + \Rc(y)$.

\begin{lemma}[Properties of proximity operators]\label{lem:prox}
Let $\Rc \in \Ga_0(\X)$ and  $\varphi \colon \R \to \R$.

\begin{enumerate}
\item \label{lem:prox1} $\prox_{\Rc} = (\id + \partial \Rc)^{-1}$.
\item \label{lem:prox2}  $\prox_{\Rc}$ is nonexpansive.
\item \label{lem:prox3}  Let $\Dc \colon \X \rightarrow \R$ be convex and Fr\'{e}chet differentiable. Then, for all $\ga \neq 0$,
\begin{equation*}
	 \argmin ( \Dc + \Rc )= \fix(\prox_{\ga \Rc} \circ (\id - \ga \nabla \Dc)) \,.
 \end{equation*}
 \item \label{lem:prox4} If $\X=\R$ then $\dom(\Rc)$ is a closed interval and $\Rc$ is continuous on $\dom(\Rc)$.

 \item \label{lem:prox5}  $\exists \qone \in \Ga_0(\R) \colon \varphi = \prox_\qone$   $\Leftrightarrow$  $\phi$ is nonexpansive and increasing.

\item  \label{lem:prox6} Let $\La$ be  at most countable and $(r_\la)_{\la \in \La} \in \Ga_0(\R)^\La$ with $\forall \la\colon r_\la \geq r_\la(0) =0$. Then $\Rc \coloneqq   \bigoplus_{\la \in \La} r_\la$ is contained in $\Ga_0(\ell^2(\La))$,  the proximity operator $\prox_\Rc$ is weakly sequentially continuous, and  $\forall (y_\la)_\la  \in \ell^2(\La) \colon \prox_\Rc ((y_\la)_\la) = (\prox_{r_\la}(y_\la))_\la$.
\end{enumerate}
\end{lemma}

Proofs of all claims in Lemma~\ref{lem:prox} can be found in \cite{bauschke2017convex}.

\begin{remark}[Proximity operators on the  real line]

Let $\qone \in \Ga_0(\R)$ with $\qone(0) = 0$. According to Lemma~\ref{lem:prox}\ref{lem:prox5}, $\prox_\qone$ is nonexpansive and increasing. Furthermore, as stated in Lemma~\ref{lem:prox}\ref{lem:prox4}, the domain of $\qone$ is a closed interval that contains zero, and  $\qone$ is continuous on the interior of $\dom(\qone)$. For all $x \in \dom(\qone)$, the subgradients at $x$ form a closed interval, and the set-valued function $x \mapsto \partial \qone(x)$ is increasing. This means that for $x \leq y$, it holds for all $u \in \partial \qone(x)$ and $v \in \partial \qone(y)$ that $u \leq v$. Hence, the same property holds for the inverse of $\prox_\qone$. When we refer to $\min$ or $\max$ of a subdifferential or the inverse of a proximity operator, we are referring to the largest lower or smallest upper bound of the respective interval.

Applying \cite[Equation (2.22)]{bauschke2017convex} with $f=1/2 \abs{\cdot}^2 + s$) one derives
\begin{equation*}
 	 s(x)= \sup_{n \geq 1} \sup_{u_1, \ldots, u_n \in \R} \left\{(x - \prox_\qone(u_n)){u_n} +
 	  \sum_{i = 0}^{n-1} (\prox_\qone(u_{i+1}) - \prox_\qone(u_i)){u_i} \right\} \\
 	 - \frac{1}{2}\abs{x}^2 \,.
\end{equation*}
While this closed form expression for $\qone$ in terms  $\prox_\qone$, this also indicated that such an relation is not very straight-forward.

Figure \ref{fig:prox} shows an example of proximity operators $\phi = \prox_s$ of a functional $\qone \in \Ga_0(\R)$ with $\qone(0) = 0$, highlighting some of its relations.  In particular, it shows how jumps of $\qone$ cause regions where $\phi$ is constant. In particular, if $s$ remains $\infty$, then the proximity operator is constant and thus not surjective.  We also see that prox is monotonically increasing and non-expansive.
\end{remark}

\begin{figure}[htb!]
\centering
\begin{tikzpicture}[>=latex, x=1cm, y=1cm, font=\small, scale=0.8]
\draw[->] (-0.6, 0) -- (3.6, 0) node[right] {$x$};
\draw[->] (0, -0.6) -- (0, 3) node[left] {$\qone(x)$};
\draw[-] (2.2, -0.2) -- (2.2, 0.2);
\draw (1.1,-0.2) node[below] {$\dom(\qone)$};
\draw[dashed, gray] (-0.5, -0.5) -- (3, 3);
\draw[scale=1, domain=-0.5:1.5, smooth, variable=\x, darkgray] plot ({\x}, {1/4*\x*\x});
\draw[scale=1, domain=1.5:2.2, smooth, variable=\x, darkgray] plot ({\x},{2*(\x-1.5)+1/4*9/4});
\draw[-, decorate, decoration={snake, amplitude=2pt}, darkgray](2.2,3)--(3.5,3);
\end{tikzpicture}
\hfill
\begin{tikzpicture}[>=latex, x=1cm, y=0.5cm, font=\small, scale=0.8]
\draw[->] (-0.6, 0) -- (3.6, 0) node[right] {$x$};
\draw[->] (0, -1.2) -- (0, 6) node[left] {$(\id+ \partial s)(x)$};
\draw[-] (2.2, -0.4) -- (2.2, 0.4);
\draw (1.1,-0.2) node[below] {$\dom(\partial\qone)$};
\draw[dashed, gray] (-0.5, -0.5) -- (3, 3);
\draw[scale=1, domain=-0.5:1.5, smooth, variable=\x, darkgray] plot ({\x}, {1/2*\x+\x});
\draw[scale=1, domain=1.5:2.2, smooth, variable=\x, darkgray] plot ({\x},{2+\x});
\draw[-,darkgray] (2.2, 4.2) -- (2.2, 6);
\draw[-, decorate, decoration={snake, amplitude=2pt}, darkgray](2.2,6)--(3.5,6);
\draw[-] (1.5, 2.25) -- (1.5, 3.5);
\end{tikzpicture}
\hfill
\begin{tikzpicture}[>=latex, x=0.5cm, y=1cm, font=\small, scale=0.8]
\draw[->] (-1.2, 0) -- (7.2, 0) node[right] {$x$};
\draw[->] (0, -0.6) -- (0, 3) node[left] {$\prox_\qone(x)$};
\draw[dashed, gray] (-0.5, -0.5) -- (3, 3);
\draw[scale=1, domain=-0.5:1.5, smooth, variable=\x, darkgray] plot ({1/2*\x+\x},\x);
\draw[scale=1, domain=1.5:2.2, smooth, variable=\x, darkgray] plot ({2+\x},{\x});
\draw[-, darkgray](4.2,2.2)--(7,2.2);
\draw[-] (2.25,1.5) -- (3.5,1.5);
\end{tikzpicture}
\caption{Example of  functional and its  proximity operator  on the real line. The dashed line always represents the identity function. The wavy line represents infinity.} \label{fig:prox}
\end{figure}
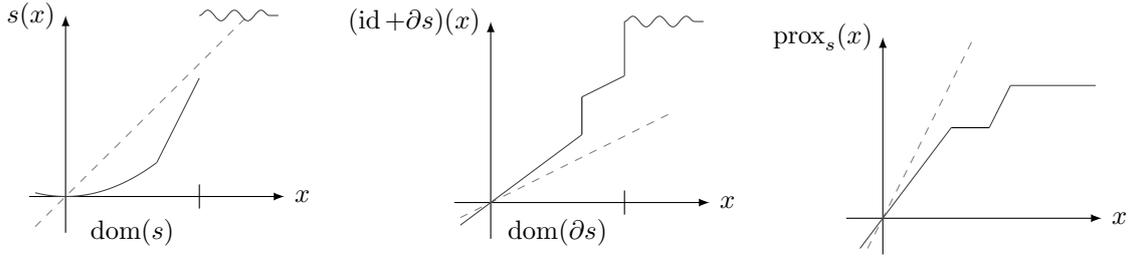

\subsection{Technical Lemmas} \label{sec:lemmas}

In this section we  provide some technical Lemmas for proximity operators on $\R$ and their connection to the associated functionals, that we use for our analysis.

\begin{lemma}\label{lem:inv-conv}
Let $(\ph_\al)_{\al>0}$ be a family of increasing and nonexpansive functions $\ph_\al \colon \R \to \R$ with $\ph_{\al}(0)=0$. The following statements are equivalent:
\begin{enumerate}[label = (\arabic*)]
\item\label{lem:inv-conv1} $\forall x \in \R \colon \lim_{\al \to 0} \ph_\al(x) = x$.
\item\label{lem:inv-conv2} $\forall x \in \R \colon \lim_{\al \to 0} \ph_\al^{-1}(x) = \{x\}$.
\end{enumerate}
Here the limit  in \ref{lem:inv-conv2} is defined by $\lim_{\al \to 0} \ph_\al^{-1}(x) = \{x\} :\Leftrightarrow \lim_{\al \to 0} \sup_{y \in \ph_\al^{-1}(x)}\abs{y-x} = 0$, and we use the convention $\sup_{y \in \emptyset}\abs{y-x} \coloneqq \infty$.
\end{lemma}

\begin{proof}
See Appendix~\ref{app:proof1}.
\end{proof}

Lemma \ref{lem:inv-conv} establishes that for a family of increasing and nonexpansive functions, the convergence of the functions to the identity function is equivalent to the convergence of their inverses to the identity function. This lemma holds practical relevance when applied to a family of proximity operators.

\begin{lemma} \label{lem:r-estimate}
Let $\qone \colon \R \to\R^+$ be convex and lower semi-continuous with $\qone(0) = 0$, and let $\ka, \al > 0$.  Suppose there exist $b,c > 0$ such that
\begin{equation} \label{eq:ass-restimate}
\forall x \in \R \colon \quad \sabs{x} \leq \min(\prox^{-1}_{\al \qone}(c \ka)) \Rightarrow \abs{\prox_{\al \qone}(x)} \leq \frac{\ka^2}{\ka^2+\al b} \sabs{x} \,.
\end{equation}
Then for all $y \in \R$  the following holds:
\begin{itemize}
\item If $\abs{y} \leq c \ka$, then  $\qone(y) \geq (b/2) \cdot \abs{ y / \ka}^2$.
\item If $\abs{y} > c \ka$, then
$\qone(y) \geq b c \abs{{y}/{\ka}}-{bc^2}/{2}$.
\end{itemize}
\end{lemma}

\begin{proof}
See Appendix~\ref{app:proof2}.
\end{proof}

\begin{lemma}\label{prox-scaled}
Let $\qone \in \Ga_0(\R)$ and $\al > 0, \ga,\ka > 0$ such that $\ga \ka^2<1$.
Suppose there exist $t \in [0,1)$ such that
\begin{equation*}
\forall x \in \R \colon \quad \sabs{x} \leq \al / \ka \Rightarrow
\sabs{\prox_{\qone}(x)} \leq \frac{\ga\ka^2 t}{1-t(1-\ga \ka^2)} \sabs{x} \,.
\end{equation*}
Then for all $x \in \R$ the following hold
\begin{enumerate}
\item
$\sabs{x} \leq \ga \al \Rightarrow \sabs{\prox_{\ga \qone(\ka (\cdot))}(x)} \leq t \sabs{x}$.
\item
$\sabs{x} > \ga \al \Rightarrow \sabs{\prox_{\ga \qone(\ka (\cdot))}(x)-x} > (1-t)\ga \al$.
\end{enumerate}
\end{lemma}

\begin{proof}
See Appendix~\ref{app:proof3}.
\end{proof}

\subsection{Diagonal frame decomposition}

 Let $\Ao \colon \X \to \Y$ be a bounded linear operator and $\La$ an at most countable index set.  The filter techniques analysed in this paper use a diagonal frame decomposition (DFD) $(\uo,\vo,\kao) = (u_\la,v_\la,\ka_\la)_{\la \in \La}$ as introduced \cite{ebner2023regularization}.

\begin{definition}[Diagonal drame decomposition, DFD] \label{def:dfd}
The triple $(\uo, \vo, \kao) = (u_\la, v_\la, \ka_\la)_{\la \in \La}$ is called  diagonal frame decomposition (DFD) for $\Ao$ if the following holds:
\begin{enumerate}[itemindent =2em, leftmargin =2.5em,  label=(DFD\arabic*)]
\item\label{DFD1}  $(u_\la)_{\la \in \La}$ is a frame for $(\ker{\Ao})^{\perp} \subseteq \X$.
\item\label{DFD2}  $(v_\la)_{\la \in \La}$ is a frame for $\overline{\ran\Ao}\subseteq \Y$.
\item\label{DFD3}  $(\kappa_\la)_{\la \in \La}\in (0, \infty)^\La$ satisfies the quasi-singular relations $\forall \la \in \La \colon  \Ao^* v_\la = \ka_\la u_\la $.\end{enumerate}
We call $(\ka_\la)_{\la \in \La}$ the family of quasi-singular values and $(u_\la)_{\la \in \La}$, $(v_\la)_{\la \in \La}$  the   corresponding quasi-singular systems.
\end{definition}

The notion of a DFD reduces to the singular value decomposition (SVD) when $(u_\la)_{\la \in \La}$ and $(v_\la)_{\la \in \La}$ are orthonormal bases. In particular, DFDs exist in quite general settings. They can also exist where there is no SVD, for example when the spectrum of $\Ao$ is continuous. However, the main advantage of DFDs over SVDs is that the quasi-singular systems can provide better approximation properties than the singular systems from the SVD. An example of this is when $(u_\la)_{\la \in \La}$ can be taken as the wavelet basis, as in the case of the Radon transform \cite{donoho1995nonlinear,ebner2023regularization,hubmer2022regularization}.

Let   $(\uo, \vo, \kao) $  be a DFD for $\Ao$ and $\bar \uo$ be a dual frame of $\uo$, defined  by  $x = \To_{\bar \uo} \To^*_{\uo} x$ for all $x \in \X$. 
Further,  let  
\begin{align*}
\To_{\bar \uo} \colon \ell^2(\La) \rightarrow \X \colon (c_\la)_{\la} \mapsto \sum_{\la \in \La} c_\la \bar u_\la \\
\To^*_{\vo} \colon \Y \rightarrow \ell^2(\La) \colon y \mapsto (\inner{y}{v_\la})_{\la \in \La},
\end{align*}
denote the synthesis and analysis operator of  $\bar \uo$ and $\vo$, respectively.  
Using the DFD, the Moore-Penrose inverse of $\Ao$ can be written as
\begin{equation} \label{eq:dfd-inv}
\forall y  \in \dom(\Ao^\mpi) \colon \quad
\Ao^\mpi (y) = \sum_{\la \in \La} \frac{1}{\ka_\la} \inner{y}{v_\la} \bar{u}_\la =  \To_{\bar \uo} \circ  \Mo^\mpi_\kao \circ \To_{\vo}^* (y)   \,,
\end{equation}
where $\Mo_{\kao}$ is the component-wise multiplication operator $\Mo_{\kao}((x_\la)_{\la \in \La}) = (\ka_\la x_\la)_{\la \in \La}$ and $\Mo^\mpi_\kao$ its Moore-Penrose inverse. Since the frame operators are continuous and invertible,  diagonalizing $\Ao$ with a DFD basically reduces the inverse problem \eqref{eq:ip} to an inverse problem with a diagonal forward operator  from $\ell^2(\La)$ to $\ell^2 (\La)$.

Due to  the ill-posedness of inverting $\Ao$, the values of $(\ka_\la)_{\la \in \La}$ accumulate at zero \cite[Thm. 7]{ebner2023regularization}, which means that $(1/\ka_\la)_\la$ is unbounded.  As a result, small errors in the data can be significantly amplified using \eqref{eq:dfd-inv}.  To reduce error amplification, we use regularizing filters aiming to damp noisy coefficients.
Unlike the widely studied linear regularizing filters  \cite{Ca02,ebner2023regularization,Gr87,hubmer2022regularization}, we investigate non-linear filters that can depend on both the operator and the data in a non-linear matter.

\section{Non-linear diagonal frame filtering}
 \label{sec:non-linear}

Throughout this paper $\X$, $\Y$ denote  Hilbert spaces and $\Ao \colon \X \rightarrow \Y$ is a bounded linear operator with an unbounded Moore-Penrose inverse  $\Ao^\mpi$.
We assume that $\Ao$ has a diagonal frame decomposition (DFD) $(\uo,\vo,\kao) = (u_\la,v_\la,\ka_\la)_{\la \in \La}$  with $\sup_{\la \in \La} \ka_\la < \infty$.

\subsection{Non-linear filtered DFD}

The following two definitions are central for this work.

\begin{definition}[Non-linear regularizing filter] \label{def:filter}
We call a family $(\ph_\al)_{\al > 0}$ of functions $\ph_{\al} \colon \R_+ \times \R \to\R$ a non-linear regularizing filter if for all $\al, \ka > 0$, the following holds
\begin{enumerate}[itemindent =2em, leftmargin =1em,  label=(F\arabic*)]
\item  \label{def:filter1} $\ph_\al(\ka, \cdot)$ is monotonically increasing.
\item  \label{def:filter2} $\ph_\al(\ka, \cdot)$ is nonexpansive.
\item  \label{def:filter3} $\ph_{\al}(\ka, 0)=0$.
\item  \label{def:filter4} $ \forall c \in \R \colon \lim_{\al \to 0} \ph_{\al}(\ka,c) = c$.
\end{enumerate}
\end{definition}

The properties required in the definition of non-linear regularizing filters are quite natural. Recall that the field of imaging often exploits the fact that natural images or signals, other than noise, are sparse in certain frames, which means that the majority of the coefficients are  zero.  Assuming the original coefficients to be small and the noisy coefficient to be zero, it makes sense to leave them at zero after filtering. Furthermore, without specific knowledge of the noise structure, it seems reasonable to preserve the order after coefficient filtering which is the  monotonicity. In addition, the distance between denoised   coefficients should not exceed the distance between the noisy coefficients, which is nonexpansive. The last property is a technical one to show the convergence of the  non-linear filtered DFD defined next to an inverse of $\Ao$.

\begin{definition}[Non-linear filtered DFD] \label{def:nl-dfd}
Let $(\ph_\al)_{\al > 0}$ be a non-linear regularizing filter.  The non-linear filtered DFD $(\Bo_\al)_{\al >0}$ with $ \Bo_\al \colon \dom(\Bo_\al) \to \X$ is defined by
\begin{align} \label{eq:non-linear-inv-dom}
	\dom(\Bo_\al)
	&\coloneqq \{y \in \Y \mid \Phi_{\al, \kao}  \circ \To_{\vo}^*  (y) \in \dom (\Mo^\mpi_\kao)\}  \,,
	\\ \label{eq:non-linear-inv}
	\Bo_\al (y)
	&\coloneqq \sum_{\la \in \La}
	\frac{1}{\ka_\la}\ph_\al(\ka_\la,\sinner{y}{v_\la}) \bar u_\la =
	\To_{\bar \uo} \circ  \Mo^\mpi_\kao \circ \Phi_{\al, \kao}
\circ \To_{\vo}^*  (y)  \,,
\end{align}
where $\Phi_{\al,\kao} \colon \ell^2(\La) \rightarrow \ell^2(\La) \colon (c_\la)_{\la \in  \La} \mapsto (\ph_{\al}(\ka_\la,c_\la))_{\la \in \La}$.
\end{definition}

According to \ref{def:filter2},  $\Phi_{\al,\kao}$ is well-defined and nonexpansive.

\begin{remark}[Linear case]
If $\ph_\al(\ka,c) = \ka f_\al(\ka) c$  for a linear regularizing filter $(f_\al)_{\al > 0}$ (as defined in \cite[Definition 3.1]{ebner2023regularization}), then $ \ph_\al$  is linear in the second component  and \eqref{eq:non-linear-inv}  reduces to the linear filtered DFD.
If $f_\al \geq 0$, $\sup\{\abs{\ka f_\al(\ka)} \mid \al, \ka> 0 \} \leq 1$ and $\lim_{\al \to 0} f_\al(\ka) = 1/\ka$, then $(\ph_{\al})_{\al >0}$ satisfies \ref{def:filter1}-\ref{def:filter4}.  Linear filtered DFDs has been analyzed and shown to be a regularization method  in \cite{ebner2023regularization}.   Here we extend the analysis to filters that are non-linear in the second component.
\end{remark}

The goal of this paper is to show stability and convergence  of  $(\Bo_\al)_{\al>0}$.
Since $\To_{\bar \uo}$ and $\To_{\vo}^*$ are continuous,  according to \eqref{eq:dfd-inv}, \eqref{eq:non-linear-inv}  it is sufficient to analyze stability and convergence of $(\Mo^\mpi_\kao \circ \Phi_{\al, \kao})_{\al>0}$.

\subsection{Filters as proximity operators}

In the next lemma we demonstrate  that non-linear regularizing filters are  proximity operators of proper, convex, and lower semi-continuous functionals.

\begin{lemma}[Filters are proximity operators]\label{lem:main}
Let $(\ph_\al)_{\al>0}$ satisfy \ref{def:filter1}-\ref{def:filter3}. Then the following hold
\begin{enumerate}
\item\label{lem:main1} There  exist $\qone_{\al,\la} \in \Gamma_0(\R)$ with $\qone_{\al,\la} \geq \qone_{\al,\la}(0)=0$ such that $\ph_{\al}(\ka_\la,\cdot) = \prox_{\qone_{\al,\la}}$.
\item\label{lem:main2}
$\Rc_\al \coloneqq \bigoplus_{\la \in \La} \qone_{\al,\la}(\ka_\la(\cdot))$ is non-negative and contained in $\Ga_0(\ell^2(\La))$ for all $\al > 0$.

\item\label{lem:main3}
$\forall (z_\la)_{\la \in \La} \in \ell^2(\La) \colon  \prox_{\Rc_\al}((z_\la)_{\la \in \La}) = (\prox_{\qone_{\al,\la}(\ka_\la (\cdot))}(z_\la))_{\la \in \La}$.

\item\label{lem:main4}  For all $z \in \dom(\Mo_{\kao}^\mpi \circ \Phi_{\al,\kao})$ and any fixed $\ga > 0$ we have
\begin{align}  \label{eq:main2}
\Mo^\mpi_\kao \circ \Phi_{\al, \kao} (z)  &= \argmin_{x  \in \ell^2(\La)}  \norm{\Mo_\kao x -z }^2/2 + \Rc_\al( x )  \,,
\\ \label{eq:main3}
\Mo^\mpi_\kao \circ \Phi_{\al, \kao} (z)   &= \fix \bigl( \prox_{\ga \Rc_\al} \circ (\id - \ga \nabla \norm{\Mo_{\kao}(\cdot) - z}^2/2) \bigr) \,.
\end{align}
\end{enumerate}
\end{lemma}

\begin{proof}
Since $\ph_{\al}(\ka_\la,\cdot) \colon \R \to\R$ is increasing and nonexpansive, according to  a Lemma~\ref{lem:prox}\ref{lem:prox5},  $\ph_{\al}(\ka_\la,\cdot) = \prox_{\qone_{\al,\la}}$ with $\qone_{\al,\la} \in \Ga_0( \R)$. By \ref{def:filter3},  $\prox_{\qone_{\al,\la}}(0)=0$ which  means $0 \in  \argmin \qone_{\al,\la}$ and thus $\qone_{\al,\la}(0) \geq \qone_{\al,\la}$.   In particular, we can choose   $\qone_{\al,\la}$ with $\qone_{\al,\la}(0)=0$ which yields  \ref{lem:main1}.   Clearly $\Rc_\al \coloneqq \bigoplus_{\la \in \La} \qone_{\al,\la}(\ka_\la(\cdot))$ is non-negative. By Lemma~\ref{lem:prox}\ref{lem:prox6},  $ \Rc_\al \in \Ga_0(\ell^2(\La))$ and $(\prox_{\qone_{\al,\la}(\ka_\la (\cdot))}(z_\la))_{\la \in \La} = \prox_{\Rc_\al}((z_\la)_{\la \in \La})$ which shows \ref{lem:main2}, \ref{lem:main3}.
One easily verifies $ \ka_\la^{-1} \, \ph_\al(\ka_\la,z_\la) = \argmin_x   \sabs{\ka_\la x -z_\la}^2/2 + \qone_{\al,\la}(\ka_\la x)  $.
Thus, for $\Mo_\kao^\mpi \circ \Phi_{\al, \kao} (z) \in \ell^2(\La)$,
\begin{align*}
\Mo_\kao^\mpi \circ \Phi_{\al, \kao} (z)
&= \bigl( \ka_\la^{-1} \,\ph_\al(\ka_\la,z_\la)\bigr)_{\la \in \La}\\
&= \bigl(\argmin_{x_\la \in \R} ( \abs{\ka_\la x_\la -z_\la }^2/2 + \qone_{\al,\la}(\ka_\la x_\la) ) \bigr)_{\la \in \La} \\
&=  \argmin_{x  \in \ell^2(\La)}  \Bigl( \sum_{\la \in  \La}  \abs{\ka_\la x_\la -z_\la }^2/2 + \qone_{\al,\la}(\ka_\la x_\la)   \Bigr) \\
&=\argmin_{x  \in \ell^2(\La)} \norm{\Mo_\kao x -z }^2/2 +\Rc_\al(x)  \,,
\end{align*}
which is \eqref{eq:main2}. Since $\sup_{\la \in \La} \ka_\la < \infty$, the operator $\Mo_{\kao}$ is bounded and thus  $ \norm{\Mo_\kao (\cdot) -z }^2/2$ is convex and Fréchet differentiable.  Together with Lemma~\ref{lem:prox}\ref{lem:prox3} we get  \eqref{eq:main3}.
\end{proof}

Lemma~\ref{lem:main} plays a central role in the upcoming convergence analysis and transforms the non-linear filtered DFD into both an equivalent optimization problem and an equivalent fixed point equation. The optimization problem consists of the least squares data fitting and the regularization function $\Rc_\al$. Unless otherwise specified, we use $\qone_{\al,\la}$ and $\Rc_\al$ to denote the functionals induced by the nonlinear regularizer filter $(\ph_\al)_{\al>0}$, according to Lemma~\ref{lem:main}.  In general, both $\Rc_\al$ and all $s_{\al,\la}$ depend on the quasi-singular values $\ka_\la$ and thus on the underlying operator. This is different from the usual variational regularization, where the regularization term is usually designed to be independent of the forward operator. With this in mind, we introduce the following notation.

\begin{definition}[$\kao$-regularizer]
Suppose $(\ph_\al)_{\al>0}$ is a non-linear regularizing filter, and let $\qone_{\al,\la} \in \Gamma_0(\R)$ with $\qone_{\al,\la} \geq \qone_{\al,\la}(0)=0$ and $\ph_{\al}(\ka_\la,\cdot) = \prox_{\qone_{\al,\la}}$ be as in Lemma~\ref{lem:main}. We call the family $(\Rc_\al)_{\al >0}$ defined by 
\begin{equation*} 
\forall (x_\la)_\la \in \ell^2(\La) \colon \quad 
 \Rc_\al((x_\la)_\la) 
\coloneqq \bigoplus_{\la \in \La} \qone_{\al,\la}(\ka_\la(x_\la))
\end{equation*}
the $\kao$-regularizer defined by the filter $(\ph_\al)_{\al>0}$ and the weight vector $\kao$.
Further, we call a $\kao$-regularizer stationary if $\Rc_\al = \al \Rc_1$ and non-stationary otherwise.  
\end{definition}

According to Lemma~\ref{lem:prox}\ref{lem:prox4},  $\Mo_\kao^\mpi \circ \Phi_{\al, \kao} (z)$ for any $z$ is contained in $\dom( \Rc_\al)$. While it is known that $x = (x_\la)_{\la} \in \dom(\Rc_\al)$ if and only if $(s_{\al,\la}(\ka_\la x_\la))_{\la} \in \ell^1(\La) $, this criterion is difficult to verify knowing only the  proximity operators. The following lemma provides a more practical condition for verifying that an element belongs to the domain of $\Rc_\al$.

\begin{lemma}\label{lem:domR}
Let $(\ph_{\al})_{\al> 0}$ be a non-linear regularizing filter and $\al>0$ be fixed.
Then, we have $\ran( \Mo^+_{\kao} \circ \Phi_{\al, \kao}) \subseteq \dom(\Rc_\al)$.
In particular, $\ran( \Mo^+_{\kao} \circ \Phi_{\al, \kao}) \subseteq \dom(\partial \Rc_\al)$.
\end{lemma}

\begin{proof}
Let $x \in \ran( \Mo^+_{\kao} \circ \Phi_{\al, \kao}) \subseteq \ell^2(\La)$ and choose $z \in \ell^2(\La)$ such that $\Mo^+_{\kao} \circ \Phi_{\al, \kao}(z)=x$.
Then, by~\eqref{eq:main2}, we have $x = \argmin_{\tilde x  \in \ell^2(\La)}  \norm{\Mo_\kao \tilde x -z }^2/2 + \Rc_\al(\tilde  x )$, 
which implies $x \in \dom(\Rc_\al)$ and consequentially by Fermat's theorem we conclude  $x \in \dom(\partial\Rc_\al)$.
\end{proof}

\section{Convergence Analysis} \label{sec:conv}

In this section we investigate stability and convergence of $(\Mo^\mpi_{\kao} \circ \Phi_{\al, \kao})_{\al>0}$.
By the continuity of the frame operators $\To_{\bar \uo}$ and $\To^*_{\vo}$, this implies stability and convergence of  
the non-linear filtered DFD 
\begin{align}
	\Bo_\al (y)
	&\coloneqq \sum_{\la \in \La}
	\frac{1}{\ka_\la}\ph_\al(\ka_\la,\sinner{y}{v_\la}) \bar u_\la =
	\To_{\bar \uo} \circ  \Mo^\mpi_\kao \circ \Phi_{\al, \kao}
\circ \To_{\vo}^*  (y)  \,,
\end{align}
rigorously defined in Definition~\ref{def:nl-dfd}. To achieve stability and convergence, we impose additional assumptions on the nonlinear regularization filter $(\ph_{\al})_{\al > 0}$. First, we simplify the filtered DFD to the familiar concept of variational regularization, which helps to understand the analytical strategies employed. In a second step, we relax the extra assumptions and analyze $(\Mo^\mpi_{\kao} \circ \Phi_{\al, \kao})_{\al>0}$ directly.

\subsection{Stationary case:  Application of variational regularization}
\label{ssec:case1}

Variational regularization uses minimizers of the generalized Tikhonov functional $\Tc_{\al,y}(x) = \norm{\Ao x -y}^2/2 + \al \Rc(x)$, where $\Rc$ is a regularizing functional and $\al > 0$ the regularization parameter.
It is  well-investigated by numerous works, such as \cite{En96,mazzieri2012existence,scherzer2009variational}.
In this paper we use the following convergence result.

\begin{lemma}[Variational regularization] \label{lem:variational}
Let $\Rc \in \Ga_0(\X)$ be norm-coercive.
Suppose  $\al_k, \al> 0$ and $y,y^k \in \Y$ with $y^k \to y$.
\begin{itemize}
\item Existence: $\Tc_{\al,y}$ has at least one minimizer.
\item Stability: For  $x^k \in \argmin \Tc_{\al,y^k}$, there exists a subsequence $(x^{k(\ell)})_{\ell \in \N}$ and some $x_\al \in \argmin \Tc_{\al,y}$ with $x^{k(\ell)} \rightharpoonup x_\al$ and $\Rc(x^{k(\ell)}) \rightarrow \Rc(x_\al)$ as $\ell  \to \infty$.
If the minimizer of $\Tc_{\al,y}$ is unique, then $x^k \rightharpoonup x_\al$ as $k \to \infty$.

\item Weak Convergence: Assume $y \in \ran(\Ao)$ and $\Ao x =y$ has a solution in $\dom(\Rc)$. Let  $\snorm{y^k-y} \leq \delta_k$  and  $\delta_k, \al_k, \delta_k^2/\al_k \to 0$. For $x^k \in \argmin \Tc_{\al_k,y^k}$ there exists a subsequence $(x^{k(\ell)})_{\ell \in \N}$ and a solution $x^\mpi$ of $\Ao x = y$ with $x^{k(\ell)} \rightharpoonup x^\mpi$ and $\Rc(x^{k(\ell)}) \rightarrow \Rc(x^\mpi)$.
If the solution is unique, then $x^k \rightharpoonup x^\mpi$.
%\item \red{Strong Convergence: Consider the situation of weak convergence with unique solution $x^+$. If $\Rc$ is totally convex at $x^+$ and $x^+ \in \dom(\partial \Rc)$, then $x^k \rightarrow x^+$ in the norm topology.}
\end{itemize}
\end{lemma}

\begin{proof}
See \cite[Theorems 3.22, 3.23, 3.26]{scherzer2009variational}.
\end{proof}

Lemma~\ref{lem:main}\ref{lem:main4} shows that the filtered DFD can be expressed as an optimization problem with  $\kao$-regularizer $\Rc_\al$. In the stationary case where $\Rc_\al = \al \Rc$ the DFD reduces to a variational regularization. We will show below that this already covers a large class of non-linear filter-based methods.  An example of this relation was given in \cite{Fr19}, where it was shown that the the soft thresholding filter yields a regularization method.

\begin{assumption}[Reduction to variational regularization]\label{ass:A}
Let $(\ph_{\al})_{\al > 0}$ be a non-linear regularizing filter which  satisfies the following additional conditions:
\begin{enumerate} [label=(A\arabic*), leftmargin =2.5em]
\item\label{ass:A1}  $\forall \ka > 0$ $\forall y \in \R$ the set $ (\ph_{\al}(\ka, \cdot)^{-1}(y) - y )/\al $ is independent of $\al > 0$.
\item\label{ass:A2} For some $\tilde\al > 0$ there exist $b,c > 0$ such that for all $\ka > 0$, we have
\begin{equation*}
	\forall x \in \R \colon \quad
	\sabs{x} \leq \min(\ph_{\tilde\al}(\ka_\la,\cdot)^{-1}(c \ka))  \Rightarrow \abs{\ph_{\tilde\al}(\ka,x)} \leq \frac{\ka^2}{\ka^2+\tilde\al b} \sabs{x} \,.
\end{equation*}
\end{enumerate}
\end{assumption}

Assumption~\ref{ass:A1} ensures stationarity and  Assumption~\ref{ass:A2}  coercivity of $\Rc$.
Furthermore, Assumption~\ref{ass:A}  implies that $\ph_{\al}(\ka,\cdot)$ is surjective. If this were not the case, there would exist  $y \in \R$ with $\ph_{\al}(\ka, \cdot)^{-1}(y) = \emptyset$ and then $\lim_{\al \to 0} \ph_{\al}(\ka, \cdot)^{-1}(y) = \emptyset$, which contradicts Lemma \ref{lem:inv-conv}.

\begin{remark}[Generation from $\ph_1$]\label{rem:case1}
Opposed to the general case, filters satisfying Assumption~\ref{ass:A} are uniquely  determined by a single filter function $\ph_1 \colon \R_+ \times \R \to\R$.
Let $\ph_1(\ka,\cdot)$ be monotonically increasing and nonexpansive  with $\ph_1 \geq \ph_1(\ka,0)=0$ for all $\ka > 0$, and assume there exist  $b, c > 0$ such that  $\abs{\ph_1(\ka,x)} \leq  \sabs{x} \, \ka^2 /(\ka^2+b) $ for all $\ka > 0$ and $x \in \R$ with $\sabs{x} \leq \min(\ph_1(\ka,\cdot)^{-1}(c\ka))$.
Then, $\ph_\al$ is uniquely defined by
\begin{equation*}
	\ph_\al(\ka, x) = ((1-\al) \id + \al \ph_1(\ka,\cdot)^{-1})^{-1}(x)  \,.
	\end{equation*}
For all $\al \in (0,1)$ the function $\ph_\al$ is well defined, and, by Lemma \ref{lem:inv-conv}, the family $(\ph_\al)_{\al > 0}$ is a non-linear regularizing filter.
\end{remark}

\begin{example} \label{ex:caseA}
For fixed  $b, d > 0$  consider the function
\begin{equation*}
\ph_1(\ka,x)= \begin{cases} \frac{\ka^2}{\ka^2 + b} \, x & \sabs{x} \leq d / \ka \\
x - \sign(x) \frac{d \, b}{\ka(\ka^2+b)} & \sabs{x} > d / \ka \,.
\end{cases}
\end{equation*}
Then, for all $\ka > 0$, $\ph_1(\ka,\cdot)$ is monotonically increasing, nonexpansive, and $\ph_1(\ka,0)=0$.
Set   $c = d / ( \max_\la  \ka_\la^2 + b)$. Then, $\ph_1(\ka_\la, \cdot)^{-1}(c \ka_\la) = ( ( b+\ka_\la^2) / \ka_\la^2) \, c \ka_\la \leq d/\ka_\la$ and by definition, $\abs{\ph_1(\ka_\la,x)} \leq  \ka_\la^2 / (\ka_\la^2+b) \sabs{x}$ for all $\sabs{x} \leq d / \ka_\la$.
By Remark \ref{rem:case1}
\[
\ph_\al(\ka,x)= \begin{cases} \frac{\ka^2}{\ka^2 + \al b} \, x & \sabs{x} \leq d \frac{\ka^2+\al b}{\ka(\ka^2+b)} \\[0.1cm]
x - \sign(x) \frac{db \al}{\ka(\ka^2+b)} & \sabs{x} > d \frac{\ka^2+\al b}{\ka(\ka^2+b)}.
\end{cases}
\]
defines a non-linear regularizing filter $(\ph_\al)_{\al > 0}$ satisfying Assumption~\ref{ass:A}.  Note that $\ph_\al$ is the proximity operator of the scaled Huber loss function $\al \cdot b/\ka^2 \cdot L( d \ka/(\ka^2 +b),x)$, where $L(\delta,x)=x^2/2 $ if $\abs{x} \leq \delta$ and $L(\delta,x)= \delta (\abs{x}-\delta/2 )$ otherwise, first introduced in \cite{Huber1964}.
\end{example}

\begin{figure}[htb!]
\centering
\begin{tikzpicture}
\draw[->] (-4, 0) -- (4, 0) node[right] {$x$};
\draw[->] (0, -3.5) -- (0, 3.5) node[left] {$\ph_\al(1/3,x)$};
\draw[dashed, gray] (-3.5, -3.5) -- (3.5, 3.5);
%\al=1 \ka=1/3
\draw[scale=1, domain=-3:3, smooth, variable=\x, black] plot ({\x}, {1/3^2/(1/3^2+1)*\x});
\draw[scale=1, domain=3:3.5, smooth, variable=\x, black] plot ({\x}, {\x- sign(\x)*3/(1/3^2+1)});
\draw[scale=1, domain=-3.5:-3, smooth, variable=\x, black] plot ({\x}, {\x- sign(\x)*3/(1/3^2+1)});
\node at (3.5,1) [right, black] {$\al=1$};
%\al=1/2 \ka=1/3
\draw[scale=1, domain=-(3/10+27/10*1/2):(3/10+27/10*1/2), smooth, variable=\x, darkgray] plot ({\x}, {1/3^2/(1/3^2+1/2)*\x});
\draw[scale=1, domain=-3.5:-(3/10+27/10*1/2), smooth, variable=\x, darkgray] plot ({\x}, {\x-sign(\x)*(1/2)*27/10});
\draw[scale=1, domain=(3/10+27/10*1/2):3.5, smooth, variable=\x, darkgray] plot ({\x}, {\x-sign(\x)*(1/2)*27/10});
\node at (3.5,2.2) [right, darkgray] {$\al=1/2$};
%\al=1/4 \ka=1/3
\draw[scale=1, domain=-(3/10+27/10*1/4):(3/10+27/10*1/4), smooth, variable=\x, gray] plot ({\x}, {1/3^2/(1/3^2+1/4)*\x});
\draw[scale=1, domain=-3.5:-(3/10+27/10*1/4), smooth, variable=\x, gray] plot ({\x}, {\x-sign(\x)*(1/4)*27/10});
\draw[scale=1, domain=(3/10+27/10*1/4):3.5, smooth, variable=\x, gray] plot ({\x}, {\x-sign(\x)*(1/4)*27/10});
\node at (3.5,2.8) [right, gray] {$\al=1/4$};
%\al=1/8 \ka=1/3
\draw[scale=1, domain=-(3/10+27/10*1/8):(3/10+27/10*1/8), smooth, variable=\x, lightgray] plot ({\x}, {1/3^2/(1/3^2+1/8)*\x});
\draw[scale=1, domain=-3.5:-(3/10+27/10*1/8), smooth, variable=\x, lightgray] plot ({\x}, {\x-sign(\x)*(1/8)*27/10});
\draw[scale=1, domain=(3/10+27/10*1/8):3.5, smooth, variable=\x, lightgray] plot ({\x}, {\x-sign(\x)*(1/8)*27/10});
\node at (3.5,3.2) [right, lightgray] {$\al=1/8$};
\end{tikzpicture}
\caption{Illustration of the filter  $(\ph_\al)_{\al > 0}$ from Example~\ref{ex:caseA} that is generated by a single filter function $\ph_1$  and satisfies Assumption~\ref{ass:A}.}
\end{figure}
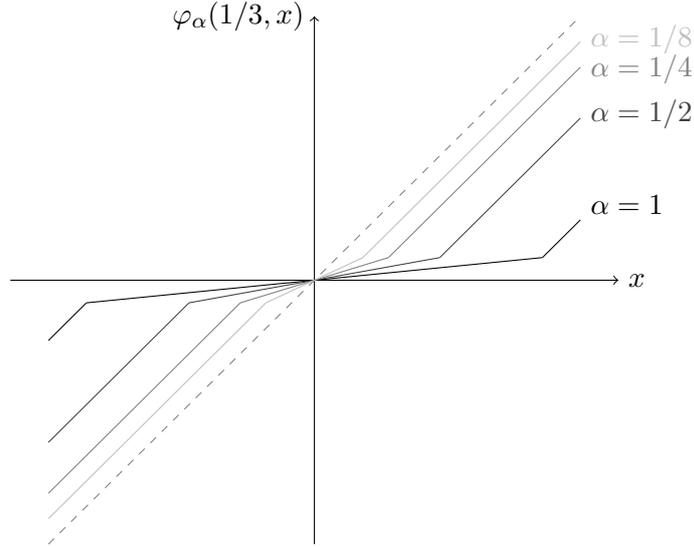

The following Lemma reduces  the non-linear filtered DFD to variational regularization.

\begin{lemma}[Reduction to variational regularization]\label{lem:var-reduction}
Let $(\ph_\al)_{\al > 0}$ be a non-linear regularizing that satisfies Assumption~\ref{ass:A}.  Then $\dom(\Mo^\mpi_\kao \circ \Phi_{\al, \kao})= \ell^2(\La)$ and there exist $\qone_\la \in \Gamma_0(\R)$ with $\qone_\la(0) = 0$ and $\ph_{\al}(\ka_\la,\cdot) = \prox_{\al \qone_\la}$. Define $\Rc = \bigoplus_{\la \in \La} \qone_\la(\ka_\la (\cdot))$, then $\Rc \in \Ga_0(\ell^2(\La))$,  $\Rc$ is norm-coercive
and
\begin{equation} \label{eq:varia}
\forall z \in \ell^2(\La) \colon \quad
\Mo^\mpi_\kao \circ \Phi_{\al, \kao} (z) = \argmin_{x \in \ell^2(\La)}
\frac{1}{2} \norm{\Mo_\kao x -z }^2 + \al \Rc(x)  \,.
\end{equation}
\end{lemma}

\begin{proof}
The identity  $\dom(\Mo^\mpi_\kao \circ \Phi_{\al, \kao})= \ell^2(\La)$ will be shown  shown under more general assumptions in Lemma \ref{prop:dom-full} below.
By Lemma~\ref{lem:main}, there exist $f_{\al, \la} \in \Gamma_0(\R)$ with $f_{\al,\la} \geq f_{\al,\la}(0) = 0$ such that $\ph_{\al}(\ka_\la, \cdot) = \prox_{ f_{\al,\la}} = (\id + \partial f_{\al,\la})^{-1}$, so
\[\partial f_{\al,\la}(y) = \ph_{\al}(\ka_\la,\cdot)^{-1}(y)-y. \]
Since $\partial (f_{\al,\la}/\al)(y)$ is independent of $\al$, we define $\qone_\la \coloneqq f_{\al,\la}/\al$ and we have $\ph_{\al}(\ka_\la,\cdot) = (\id + \partial (\al \qone_\la))^{-1} = \prox_{\al \qone_\la}$.
Note that $\qone_\la$ is also proper, convex, and lower semi-continuous and $\qone_\la(0)=0$. Moreover, $\Rc=\bigoplus_{\la \in \La} \qone_\la(\ka_\la (\cdot))$ is  positive, $\Rc \in \Ga_0(\ell^2(\La))$ and satisfies \eqref{eq:varia}.
To show the norm-coercivity of $\Rc$ consider Lemma~\ref{lem:r-estimate}.
For ${y = (y_\la)_\la \in \ell^2(\La)}$,
\begin{multline*}
\Rc(y)  = \sum_{\la \in \La} \qone_\la(\ka_\la y_\la)
 = \sum_{\sabs{\ka_\la y_\la} \leq c \ka_\la} \qone_\la(\ka_\la y_\la) + \sum_{\sabs{\ka_\la y_\la} > c \ka_\la} \qone_\la(\ka_\la y_\la) \\
 \geq \sum_{\sabs{y_\la} \leq c} b\sabs{y_\la}^2/2 + \sum_{\sabs{y_\la} > c} \bigl(b c \sabs{y_\la}-  b c^2 /2 \bigr)
  \geq \frac{b}{2} \sum_{\sabs{y_\la} \leq c} \sabs{y_\la}^2 + (b c^2 /2) \abs{\{ \la \mid \sabs{y_\la} > c\}}.
\end{multline*}
Now let $(y^k)_k \in \dom(\Rc)^\N$ be a sequence with $\snorm{y^k}_2 \to \infty$ as $k \to \infty$ and define $N_k \coloneqq \sabs{\{\la \mid \sabs{y_\la^k} > c \}}$.  We now show that  $(y^k)_k$ can be covered by subsequences $\tau_i$ with  $\Rc(y^{\tau_i(k)})_k \to \infty$ for $i=1,2,3$.
For all subsequences $(\tau_1(k))_k$ with $\norm{y^{\tau_1(k)}}_\infty \to \infty$ we have for $k$ large enough that
\[ \Rc ( y^{\tau_1(k)}) \geq bc \sum_{\sabs{y_\la^{\tau_1(k)}} > c} \Bigl(\sabs{y_\la^{\tau_1(k)}}-\frac{c}{2}\Bigr) \geq bc \Bigl(\snorm{y^{\tau_1(k)}}_\infty -\frac{c}{2}\Bigr)  \to \infty. \]
For subsequences $(\tau_2(k))_k$ with $N_{\tau_2(k)} \to \infty$ we have $\Rc ( y^{\tau_2(k)}) \geq bc^2 N_{\tau_2(k)} /2  \to \infty$.
Finally, the rest are subsequences $(\tau_3(k))_k$, where there exists a constant $M> 0$ such that $\snorm{y^{\tau_3(k)}}_\infty \leq M$ and $N_{\tau_3(k)}$ is bounded, then
\[ \Rc(y^{\tau_3(k)}) \geq \frac{b}{2} \biggl( \snorm{y^{\tau_3(k)}}_2^2 - \sum_{\sabs{y_\la^{\tau(k)}} > c } \sabs{y_\la^{\tau(k)}}^2 \biggr)
\geq \frac{b}{2} \snorm{y^{\tau_3(k)}}_2^2 - \frac{b}{2}M N_{\tau_3(k)} \to \infty. \]
Hence $\Rc(y^k) \to \infty$, which shows that $\Rc$ is coercive.
\end{proof}

Due to Assumption~\ref{ass:A}, the operator $(\Phi_{\al,\kao}^{-1} - \id)/\al$ is independent of $\al$ and by Lemma \ref{lem:domR} we have
\begin{equation*}
	\dom\left( \frac{\Phi_{\al,\kao}^{-1} - \id}{\al} \circ \Mo_{\kao}\right) \subseteq \dom(\partial \Rc) \subseteq \dom(\Rc) \,.
\end{equation*}
From Lemmas \ref{lem:var-reduction} and \ref{lem:variational}  we obtain the following.

\begin{proposition}[Existence, Stability, Weak Convergence]\label{prop:caseA}
Let $(\ph_\al)_{\al > 0}$ be a non-linear regularizing filter such that Assumption~\ref{ass:A} is satisfied.
Suppose $\al > 0$, and $z, z^k \in \ell^2(\La)$ for $k \in \N$ with $(z^k)_{k \in \N} \to z$  and $c_\al \coloneqq  \Mo_{\kao}^\mpi \circ \Phi_{\al, \kao} (z)$.

\begin{itemize}
\item Existence: $ \dom(\Mo_{\kao}^\mpi \circ \Phi_{\al, \kao}) = \ell^2(\La)$.

\item Stability:  With $c^k \coloneqq  \Mo_{\kao}^\mpi \circ \Phi_{\al, \kao} (z^k)$ we have $c^k \rightharpoonup c_\al$ and $\Rc(c^k) \rightarrow \Rc(c_\al)$ as $k \to \infty$.

\item Weak Convergence:  Let $z =\Mo_{\kao}c^\mpi \in \dom\left( \frac{\Phi_{\al,\kao}^{-1} - \id}{\al}\right)$ with  $\snorm{z^k -z} \leq \delta_k$ where $\delta_k \to 0$.
Consider $\al_k \to 0$ such that $\delta_k^2/\al_k \to 0$ and define $c^k = \Mo_{\kao}^\mpi \circ \Phi_{\al_k, \kao} (z^k)$.
Then ${c^k  \rightharpoonup c^\mpi}$
and $\Rc(c^k) \rightarrow \Rc(c^\mpi)$  as $k \to \infty$.
\end{itemize}
\end{proposition}

\begin{proof}
Follows from  Lemmas \ref{lem:variational} and \ref{lem:var-reduction} and the injectivity of $\Mo_{\kao}$.
\end{proof}

Additional results in variational regularization are known, indicating convergence in the norm topology and convergence rates under more stringent assumptions. However, the exploration of the connection between these assumptions and filter-based methods falls outside the scope of this paper. The detailed analysis of these additional results is reserved for future research.

\begin{remark}[Strong Convergence]
In the context of Proposition \ref{prop:caseA}, under the assumption that $\Rc$ is totally convex at $c^+$ as established in Proposition 3.32 of \cite{scherzer2009variational}, the sequence $c^k$ converges to $c^+$ in the norm topology. 
\end{remark}
\begin{remark}[Convergence Rates]
Utilizing Theorem 3.42 of \cite{scherzer2009variational}, we can also establish convergence rates under a suitable source condition.
In the context of Proposition \ref{prop:caseA}, assuming that $\Rc$ is totally convex at $c^+$ and that there exists a collection $\omega_\la \in \partial s_\la(\ka_\la c^+_\la)$ such that $\omega = (\omega_\la)_{\la \in \La} \in \ell^2(\La)$ the following inequalities hold:
\[D_{\Mo_{\kao}\omega}(c^k,c^+) \leq C \delta_k \quad \text{and} \quad \norm{\Mo_{\kao}c^k- z^k}\leq \tilde C \delta_k,
\]
where $D_{\Mo_{\kao}\omega}(c^k,c^+)\coloneqq \Rc(c^k)-\Rc(c^+) - \inner{\Mo_{\kao}\omega}{c^k-c^+}$ represents the Bregman-distance of $\Rc$ at $c^+$ and $\Mo_{\kao}\omega \in \partial \Rc(c^+)$ (see \cite{Bu04}).
\end{remark}

\subsection{Direct analysis in the general case}\label{sec:case3}
\label{ssec:case2}

Assumption~\ref{ass:A1} imposes that $\Rc_\al$ must take the form $\al \Rc$, though such a constraint is not essential.
In this section, we relax the additional conditions on $(\ph_\al)_{\al > 0}$. While using approaches similar to variational regularization, we focus on the diagonal structure of $\Mo_{\kao} \circ \Phi_{\al,\kao}$ instead of relying on linear dependence with $\al$.

\begin{assumption} \label{ass:B}
Let $(\ph_{\al})_{\al > 0}$ be a non-linear regularizing filter which  satisfies the following conditions:
\begin{enumerate}[label=(B\arabic*), leftmargin =2.5em]
\item \label{ass:B1} $\forall \ka > 0$  $\forall x \in \R \colon (\abs{\ph_\al(\ka,x)})_{\al > 0}$ is monotonically increasing for $\al\downarrow 0$.
\item \label{ass:B2} $ \exists d, e> 0$ $\forall \ka > 0$  $ \forall \al > 0$ $ \forall x \in \R \colon
\sabs{x} \leq d \al / \ka \Rightarrow \abs{\ph_{\al}(\ka,x)} \leq \frac{e \ka}{\sqrt{\al}} \sabs{x}$.
\end{enumerate}
\end{assumption}

Next we show that Assumption~\ref{ass:A} is weaker than Assumption~\ref{ass:B}, hence the results  in this section generalize the results  of variational regularization.

\begin{lemma}
If the non-linear regularizing filter $(\ph_{\al})_{\al > 0}$ satisfies Assumption~\ref{ass:A} with constants $b,c >0$, then it satisfies Assumption \ref{ass:B} with $d= bc$ and $e=1/(2\sqrt{b})$.
\end{lemma}

\begin{proof}
Let $(\ph_{\al})_{\al > 0}$ satisfy Assumption~\ref{ass:A} with constants $b,c >0$ and let $\al > 0$, $\ka> 0$ and $y \in \R$ be arbitrary. By Lemma \ref{lem:var-reduction} and Lemma \ref{lem:prox} we have $\ph_{\al}(\ka,\cdot)^{-1}(y) = y + \al \partial \qone(y)$, where $\qone \in \Ga_0(\R)$.
Since $\partial \qone (y)$ is a closed interval, the maximum value of $\abs{\partial \qone (y)}$ exists and $\max \sabs{\ph_{\al}(\ka,\cdot)^{-1}(y)} < \max \sabs{\ph_{\bar\al}(\ka,\cdot)^{-1}(y)}$ for all $\al < \bar{\al}$.
Therefore, $\abs{\ph_{\al}(\ka,y)} < \abs{\ph_{\tilde\al}(\ka,y)}$ which shows \ref{ass:B1}.
Choose $\tilde\al > 0$ such that $\abs{\ph_{\tilde\al}(\ka,x)} \leq \sabs{x} \ka^2 / (\ka^2 + \tilde\al b) $ for all $x \in \R$ with $\sabs{x} \leq \min \ph_{\tilde\al}(\ka,\cdot)^{-1}(c \ka) = c\ka + \tilde\al \min \partial \qone(c \ka)$.
Following the proof of Lemma \ref{lem:r-estimate}, one shows that if $b$ and $c$ are suitable constants to satisfy \ref{ass:A2} for one $\tilde\al > 0$, they are for all $\al > 0$.
Now set $d \coloneqq bc$ and $e \coloneqq 1/(2 \sqrt{b})$. Then,
\[
\abs{\ph_{\al}(\ka,x)} \leq \frac{\ka^2}{\ka^2 + \al b} \sabs{x} = \frac{\sqrt{\al} \ka}{\ka^2 + \al b} \cdot \frac{\ka}{\sqrt{\al}} \sabs{x} \leq e \frac{\ka}{\sqrt{\al}} \sabs{x}
\]
for all $x \in \R$ with $\sabs{x} \leq c \ka + \al \min \partial \qone(c\ka)$, whereas $\min \partial \qone(c\ka) \geq d/\ka$. Hence  the inequality holds for $\sabs{x} \leq \al d / \ka$, showing \ref{ass:B2}.
\end{proof}

\begin{example}\label{ex:caseB}
Fix the constants $b, d > 0$ and consider the function
\begin{equation*}
\ph_\al(\ka,x)= \begin{cases} \frac{\ka^2}{\ka^2 + \al b} \, x & \sabs{x} \leq d \al /  \ka \\
x - \sign(x) \frac{db \al^2}{\ka(\ka^2+\al b)} & \sabs{x} > d \al /  \ka   \,.
\end{cases}
\end{equation*}
$(\ph_\al)_{\al > 0}$ is a non-linear regularizing filter and satisfies Assumption \ref{ass:B}.
Note that $\ph_1$ is the same function as in Example~\ref{ex:caseA} but for $\al \neq 1$ the filter functions differ.
By Remark~\ref{rem:case1} the construction of the non-linear regularizing filter satisfying Assumption~\ref{ass:A} is uniquely given by $\ph_1$. Thus, $(\ph_\al)_{\al > 0}$ does not satisfy Assumption~\ref{ass:A}. Especially, $\ph_\al$ is the proximity function of the Huber function $\al \cdot b/\ka^2 \cdot L(d \al \ka/(\ka^2+\al b),x)$, where we clearly see that the resulting regularizing functional is non-stationary.  In Section \ref{sec:num}, we compare this regularizing filter with that presented in  in Example \ref{ex:caseA} and the soft thresholding filter in terms of numerical rates.
\end{example}

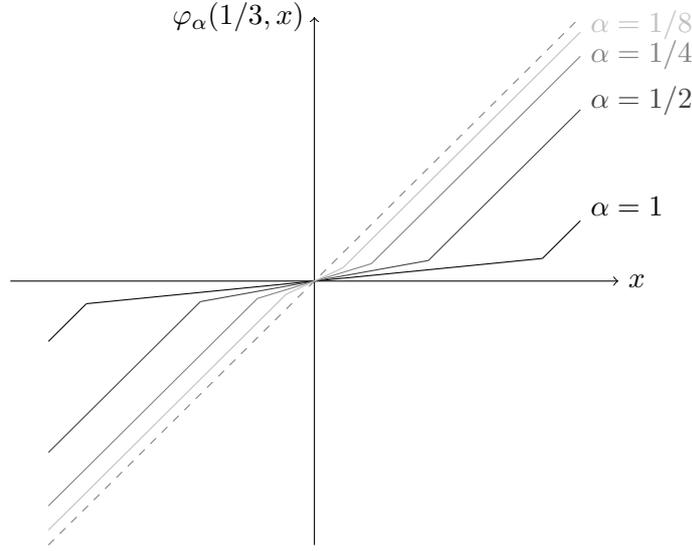
\begin{figure}[htb!]
\centering
\begin{tikzpicture}
\draw[->] (-4, 0) -- (4, 0) node[right] {$x$};
\draw[->] (0, -3.5) -- (0, 3.5) node[left] {$\ph_\al(1/3,x)$};
\draw[dashed, gray] (-3.5, -3.5) -- (3.5, 3.5);
%\al=1 \ka=1/3
\draw[scale=1, domain=-3:3, smooth, variable=\x, black] plot ({\x}, {1/3^2/(1/3^2+1)*\x});
\draw[scale=1, domain=3:3.5, smooth, variable=\x, black] plot ({\x}, {\x- sign(\x)*3/(1/3^2+1)});
\draw[scale=1, domain=-3.5:-3, smooth, variable=\x, black] plot ({\x}, {\x- sign(\x)*3/(1/3^2+1)});
\node at (3.5,1) [right, black] {$\al=1$};
%\al=1/2 \ka=1/3
\draw[scale=1, domain=-(3*1/2):(3*1/2), smooth, variable=\x, darkgray] plot ({\x}, {1/3^2/(1/3^2+1/2)*\x});
\draw[scale=1, domain=-3.5:-(3*1/2), smooth, variable=\x, darkgray] plot ({\x}, {\x-sign(\x)*(1/2)^2*27/(1+9*(1/2))});
\draw[scale=1, domain=(3*1/2):3.5, smooth, variable=\x, darkgray] plot ({\x}, {\x-sign(\x)*(1/2)^2*27/(1+9*(1/2))});
\node at (3.5,2.4) [right, darkgray] {$\al=1/2$};
%\al=1/4 \ka=1/3
\draw[scale=1, domain=-(3*1/4):(3*1/4), smooth, variable=\x, gray] plot ({\x}, {1/3^2/(1/3^2+1/4)*\x});
\draw[scale=1, domain=-3.5:-(3*1/4), smooth, variable=\x, gray] plot ({\x}, {\x-sign(\x)*(1/4)^2*27/(1+9*1/4)});
\draw[scale=1, domain=(3*1/4):3.5, smooth, variable=\x, gray] plot ({\x}, {\x-sign(\x)*(1/4)^2*27/(1+9*1/4)});
\node at (3.5,3) [right, gray] {$\al=1/4$};
%\al=1/8 \ka=1/3
\draw[scale=1, domain=-(3*1/8):(3*1/8), smooth, variable=\x, lightgray] plot ({\x}, {1/3^2/(1/3^2+1/8)*\x});
\draw[scale=1, domain=-3.5:-(3*1/8), smooth, variable=\x, lightgray] plot ({\x}, {\x-sign(\x)*(1/8)^2*27/(1+9*1/8)});
\draw[scale=1, domain=(3*1/8):3.5, smooth, variable=\x, lightgray] plot ({\x}, {\x-sign(\x)*(1/8)^2*27/(1+9*1/8)});
\node at (3.5,3.4) [right, lightgray] {$\al=1/8$};
\end{tikzpicture}
\caption{Illustration of the filter of Example~\ref{ex:caseB} that satisfies Assumption~\ref{ass:B} but not Assumption~\ref{ass:A}.}
\end{figure}

\begin{proposition}[Existence]\label{prop:dom-full}
Let $(\ph_{\al})_{\al > 0}$ be non-linear regularizing filter such that \ref{ass:B2} is satisfied.
Then, $\dom(\Mo^\mpi_{\kao} \circ \, \Phi_{\al,\kao} )= \ell^2(\La) $.
\end{proposition}

\begin{proof}
Let $\al > 0$,  $x = (x_\la)_{\la \in \La} \in \ell^2(\La)$ and $d, e > 0$ such that $\forall\al > 0\colon
\abs{\ph_{\al}(\ka_\la,z)} \leq \sabs{z} e \ka_\la/\sqrt{\al} $ for all $z\in \R$ with $\sabs{z} \leq d \al / \ka_\la$.
Since $\sup_{\la \in \La} \ka_\la < \infty$ there are only finitely many $\la$ such that $\sabs{x_\la} > d \al / \ka_\la$. Thus,
\begin{multline*}
\norm{\Mo^\mpi_{\kao} \circ \, \Phi_{\al,\kao}(x)}^2 = \sum_{\la \in \La} \frac{\sabs{\ph_{\al}(\ka_\la,x_\la)}^2}{\ka_\la^2} \\
= \sum_{\sabs{x_\la}\leq d \al / \ka_\la } \frac{\sabs{\ph_{\al}(\ka_\la,x_\la)}^2}{\ka_\la^2}+ \sum_{\sabs{x_\la} > d \al / \ka_\la} \frac{\sabs{\ph_{\al}(\ka_\la,x_\la)}^2}{\ka_\la^2}
\leq \sum_{\sabs{x_\la}\leq d \al /\ka_\la } \frac{e^2 }{\al}\sabs{x_\la}^2 + C < \infty \,.
\end{multline*}
\end{proof}

Under Assumption \ref{ass:B}, we can show that the mappings $\Mo_{\kao} \circ \Phi_{\al,\kao}$ are strong-weak continuous and converges to $\Mo_{\kao}$ as $\al \to 0$.

\begin{proposition}[Stability]\label{prop:stabC}
Let $(\ph_{\al})_{\al > 0}$ be non-linear regularizing filter such that Assumption \ref{ass:B} is satisfied.
Let $\al > 0$ be fixed and $z,z^k  \in \ell^2(\La)$  such that $\snorm{z-z^k}_2 \to 0$.  Then $\Mo^\mpi_{\kao} \circ \Phi_{\al,\kao} (z^k)  \rightharpoonup \Mo^\mpi_{\kao} \circ \, \Phi_{\al,\kao} (z)$ as  $k \to \infty$.
\end{proposition}

\begin{proof}
Fix the constant $d$ in Assumption \ref{ass:B} and  set $\La_k \coloneqq \{ \la \in \La \mid \sabs{z^k_\la} > d\al /\ka_\la  \}$ for $k \in \N$. Then, for all $\la \in \La_k$,
\[
\ka_{\la}\geq \frac{d \al}{\sabs{z^k_\la}} \geq\frac{d \al}{\sabs{z^k_\la}+\delta_k} \geq \frac{d \al}{\snorm{z}_\infty + \max_k \delta_k} \eqqcolon a > 0 \,.
\]
Since $a$ is independent of $k$, it follows  $\inf_{\la \in \La_k}\ka_\la \geq a$ for all $k \in \N$.
Therefore, defining $x^k \coloneqq  \Mo^\mpi_{\kao} \circ \Phi_{\al,\kao} (z^k)$,
\begin{align*}
\snorm{x^k}^2 &= \sum_{\ka_\la \geq a} \Bigl\lvert \frac{1}{\ka_\la}\ph_{\al}(\ka_\la,z^k_\la)\Bigr\rvert ^2 + \sum_{\ka_\la < a} \Bigl\lvert \frac{1}{\ka_\la}\ph_{\al}(\ka_\la,z^k_\la)\Bigr\rvert ^2 \\
& \leq \frac{1}{a^2} \sum_{\ka_\la \geq a} \sabs{z^k_\la}^2 + \sum_{\ka_\la < a} \frac{e^2}{\al} \sabs{z^k_\la}^2 \\
& \leq \Bigl( \frac{1}{a^2}+\frac{e^2}{\al}\Bigr)\cdot  (\snorm{z}+\max_k \delta_k)^2 \,.
\end{align*}
Hence $(x^k)_k$ is bounded and has a weakly convergent subsequence $(x^{n(\ell)})_\ell$.
By Lemma \ref{lem:main}, there exist $\qone_{\al,\la} \in \Ga_0(\R)$ with $\qone_{\al,\la}(0) = 0$, $\ph_{\al}(\ka_\la, \cdot) = \prox_{\qone_{\al,\la}}$, $\Rc_\al \coloneqq \bigoplus_{\la \in \La} \qone_{\al,\la}(\ka_\la (\cdot)) \in \Ga_0(\ell^2(\La))$ positive and $(\prox_{\ga \qone_{\al,\la}(\ka_\la (\cdot))}(x_\la))_{\la \in \La} = \prox_{\ga \Rc_\al}((x_\la)_{\la \in \La})$. Thus,
\begin{multline*}
 x^k  = \fix \bigl(\prox_{\ga \Rc_\al} \circ ( \id - \ga \nabla \snorm{\Mo_{\kao}(\cdot) - z^k}^2 /2 ) \bigr) \\
=  \bigl( \fix \bigl( \prox_{\ga \qone_{\al,\la}(\ka_\la (\cdot))} ((1 - \ga \ka_\la^2)(\cdot) + \ga \ka_\la z_\la^k) \bigr)\bigr)_{\la} \,.
\end{multline*}
Assume $x^{n(\ell)} \rightharpoonup x^\mpi$ and $y \in \ell^2(\La)$. With Lemma~\ref{lem:prox}\ref{lem:prox6},
\begin{align*}
& \sabs{\sinner{\prox_{\ga \Rc_\al} \circ \left(\id - \ga \Mo_{\kao}^2 + \ga \Mo_{\kao} z \right)(x^\mpi) - x^\mpi}{y}} \\
& \leq \sabs{\sinner{\prox_{\ga \Rc_\al} \circ (\id - \ga \Mo_{\kao}^2 + \ga \Mo_{\kao} z )(x^\mpi) - \prox_{\ga \Rc_\al} \circ (\id - \ga \Mo_{\kao}^2 + \ga \Mo_{\kao} z^{n(\ell)} )(x^\mpi)}{y}} \\
&+ \sabs{\sinner{\prox_{\ga \Rc_\al} \circ (\id - \ga \Mo_{\kao}^2 + \ga \Mo_{\kao} z^{n(\ell)})(x^\mpi) - \prox_{\ga \Rc_\al} \circ (\id - \ga \Mo_{\kao}^2 + \ga \Mo_{\kao} z^{n(\ell)} )(x^{n(\ell)})}{y}} \\
&+ \sabs{\sinner{x^{n(\ell)} - x^\mpi}{y}}  \to 0.
\end{align*}
Thus, $x^\mpi = \Mo^\mpi_{\kao} \circ \, \Phi_{\al,\kao}(z)$.
\end{proof}

\begin{proposition}[Weak Convergence]\label{prop:stabC}
Let $(\ph_{\al})_{\al > 0}$ be non-linear regularizing filter such that  Assumption \ref{ass:B} is satisfied.
Let $z\in \ran(\Mo_{\kao}) \cap \ran(\Phi_{\tilde \al, \kao})$ for some $\tilde{\al}> 0$.
Let $z^k  \in \ell^2(\La)$ such that $\snorm{z-z^k}_2 \leq \delta_k$ and let $\delta_k, \al_k \to 0$ such that $\al_k \gtrsim \delta_k^2$. Then $\Mo^\mpi_{\kao} \circ \Phi_{\al_k,\kao} (z^k) \rightharpoonup \Mo^\mpi_{\kao} z$.
\end{proposition}

\begin{proof}
Fix the constants $d$ and $e$ of Assumption \ref{ass:B} and define $\La_k \coloneqq \{\la \in \La \mid \abs{z_\la^k} \geq d \al_k/(\ka_\la) \}$ for $n \in \N$.
Then, $\snorm{z}_\infty + \max_k \delta_k \geq \sabs{z_\la}+\delta_k \geq \sabs{z^k_\la} \geq d \al_k/\ka_{\la} \geq C \delta_k/ \ka_\la$ for all  $\la \in \La_k$ and some constant $C > 0$.
Thus $(\delta_k / \inf_{\la \in \La_k} \ka_\la)_{k \in \N}$ is bounded, as otherwise one could create a contradiction to $z \in \ell^2(\La)$.
Therefore, for some constant $D > 0$,

\begin{align*}
\snorm{x^k}^2 &= \sum_{\la\in \La_k} \Bigl\lvert \frac{1}{\ka_\la}\ph_{\al_k}(\ka_\la,z^k_\la)\Bigr\rvert^2 + \sum_{\underset{\ka_\la^2 > \al_k}{\la \notin \La_k}} \Bigl\lvert\frac{1}{\ka_\la}\ph_{\al_k}(\ka_\la,z^k_\la)\Bigr\rvert^2 + \sum_{\underset{\ka_\la^2 \leq\al_k}{\la \notin \La_k}} \Bigl\lvert\frac{1}{\ka_\la}\ph_{\al_k}(\ka_\la,z^k_\la)\Bigr\rvert^2 \\
& \leq \sum_{\la\in \La_k} \left(\frac{\sabs{z^k_\la}}{\ka_\la}\right)^2 + \sum_{\underset{\ka_\la^2 > \al_k}{\la \notin \La_k}} \left(\frac{\abs{z^k_\la}}{\ka_\la}\right)^2 + e \sum_{\underset{\ka_\la^2 \leq \al_k}{\la \notin \La_k}} \left(\frac{\abs{z^k_\la}}{\sqrt{\al_k}}\right)^2 \\
& \leq \sum_{\la\in \La_k} \left( \frac{\sabs{z_\la} + \abs{z_\la^k-z_\la}}{\ka_\la}\right)^2 + \sum_{\underset{\ka_\la^2 > \al_k}{\la \notin \La_k}} \left(\frac{\abs{z_\la}}{\ka_\la} + \frac{\abs{z_\la^k-z_\la}}{\sqrt{\al_k}}\right)^2 + e \sum_{\underset{\ka_\la^2 \leq \al_k}{\la \notin \La_k}} \left(\frac{\abs{z_\la}}{\ka_\la} + \frac{\abs{z_\la^k-z_\la}}{\sqrt{\al_k}}\right)^2 \\
& \leq \left(\norm{\Mo_{\kao}^+ z} + \frac{\delta_k}{\inf_{\la \in \La_k} \ka_\la} \right)^2 + \left(\norm{\Mo_{\kao}^+ z} +\frac{\delta_k}{\sqrt{\al_k}}  \right)^2  + e \left(\norm{\Mo_{\kao}^+ z} + \frac{\delta_k}{\sqrt{\al_k}}  \right)^2\\
& \leq D.
\end{align*}
Thus $(x^k)_k$ is bounded and has a weakly convergent subsequence.

Now $\la \in \La$. By Lemma \ref{lem:inv-conv}, we have $\ph_{\al}(\ka_\la,\cdot)^{-1}(y) = y + \partial \qone_{\al,\la}(y) \rightarrow \{y\}$ and therefore $\partial \qone_{\al,\la}(y) \rightarrow \{0\}$ as $\al \to 0$ for all $y \in \R$.
Let $y \in \R$ be arbitrary and for every $\al > 0$ choose $z_\al \in \partial \qone_{\al,\la}(y)$.
The definition of the sub-differential implies that $\qone_{\al,\la}(y) \leq z_\al y$.
Since $\qone_{\al,\la}$ is positive and $z_\al \to 0$ as $\al \to 0$, we have that $\qone_{\al,\la}(y) \to 0$.
Consequently, for a fixed $\la \in \La$, $\qone_{\al, \la}$ converge point-wise to the zero function.

Further, by the point-wise monotonicity of $\ph_\al$ as $\al \rightarrow 0$ and the fact that $\qone_{\al,\la}(0) =0$ for all $\al > 0$, we get that $\qone_{\al,\la}(y)$ is monotonically decreasing as $\al \rightarrow 0$.
Then the theorem of monotone convergence implies $\Rc_\al(\Mo_{\kao}^\mpi z) = \sum_{\la \in \La} \qone_{\al,\la}(z_\la) \to 0$ as $\al \to 0$.
In particular, $\Mo^\mpi_{\kao} z \in \ran(\Mo^+_{\kao}  \circ \Phi_{\al,\kao}) \subseteq \dom(\Rc_{\al})$ for all $ \al \in (0, \tilde{\al})$.

Let $(x^{n(\ell)})_\ell$ be a weakly convergent subsequence with weak limit $x^\mpi$.
By Lemma \ref{lem:main}, we have $x^{n(\ell)} \in \argmin_x \norm{\Mo_{\kao} x - z^{n(\ell)}}^2/2 + \Rc_{\al_{n(\ell)}}(x)$ and, as $\ell \to \infty$,
\begin{multline*}
\frac{1}{2} \bigl\lVert \Mo_{\kao} x^{n(\ell)} - z^{n(\ell)} \bigr\rVert^2 + \Rc_{\al_{n(\ell)}}(x^{n(\ell)})
 \leq \frac{1}{2}\bigl\lVert z - z^{n(\ell)}\bigr\rVert^2 + \Rc_{\al_{n(\ell)}}(\Mo^\mpi_{\kao} z) \\
 \leq \delta_{n(\ell)}^2/2 + \Rc_{\al_{n(\ell)}}(\Mo^\mpi_{\kao} z)
\to 0 \,.
\end{multline*}
Thus  $\snorm{\Mo_{\kao} x^{n(\ell)} - z^{n(\ell)}} \to 0$ and since $z^{n(\ell)} \to z$ and $\Mo_{\kao}$ is linear and bounded, we conclude $\Mo_{\kao} x^{n(\ell)} \rightharpoonup \Mo_{\kao} x^\mpi$.
Therefore, $z= \Mo_{\kao} x^\mpi$ and $x^\mpi = \Mo_{\kao}^\mpi z$.
Because this holds for every weakly convergent subsequence, we conclude $x^k \rightharpoonup \Mo_{\kao}^\mpi z$.
\end{proof}

\subsection{Main theorems}

Utilizing the results obtained for $\Mo^\mpi_{\kao} \circ \Phi_{\al_k,\kao}$ under Assumption~\ref{ass:B}, we can deduce stability and weak convergence for the non-linear filter-based reconstruction method $(\Bo_\al)_{\al >0}$. Notably, these results persist under the stricter Assumption \ref{ass:A}.

\begin{theorem}[Stability] \label{thm:stab}
Let $(\ph_{\al})_{\al > 0}$ be a non-linear regularizing filter, $\al > 0$ and $y^\delta, y^k \in \Y$ be with  $y^k \to y^\delta$. If Assumption \ref{ass:B} holds, then $\Bo_\al (y^k) \rightharpoonup \Bo_\al (y^\delta)$.
\end{theorem}

\begin{proof}
The stability, in fact, follows from the representation $\Bo_\al = \To_{\bar \uo} \circ  \Mo^\mpi_\kao \circ \Phi_{\al, \kao}
\circ \To_{\vo}^*$ and the stability results in Propositions \ref{prop:stabC} combined  with the continuity of $\To_{\bar \uo}$ and $\To_{\vo^*}$.
\end{proof}

\begin{theorem}[Weak Convergence] \label{thm:conv}
Let $(\ph_{\al})_{\al > 0}$ be a non-linear regularizing filter satisfying Assumption \ref{ass:B}, $(\delta_k)_k, (\al_k)_k$  null sequences with $\delta_k^2/\al_k \to 0$ and $y \in \ran(\Ao)$.
Suppose $\To_{\vo}^*(y) \in \ran(\Phi_{\tilde \al,\kao})$ for some $\tilde{\al}> 0$ and let $(y^k)_k \in \Y^\N$ with $\snorm{y^k -y} \leq \delta_k$.
Then, $\Bo_{\al_k} (y^k)   \rightharpoonup \Ao^\mpi y$ as $k \to \infty$.
\end{theorem}

\begin{proof}
Write  $x^\mpi = \Ao^\mpi y$ and $c^k = \Mo_\kao^\mpi \circ \Phi_{\al_k, \kao} \circ \To_\vo^* (y^k)$. By the definition  of the  DFD we have
\begin{equation*}
    \snorm{\To_\vo^* y_k-\Mo_\kao \To_\uo^* x^\mpi} = \norm{\To_\vo^* y_k - \To_\vo^* \Ao x^\mpi} \leq \norm{\To_\vo}\delta_k \,.
\end{equation*}
Application of the convergence results of Propositions \ref{prop:stabC}  shows $c^k \rightharpoonup \Mo_\kao^\mpi \Mo_\kao \To_\uo^*x^\mpi = \To_\uo^*x^\mpi$.
Thus for any $z \in \X$, we have
\begin{equation*}
 \sinner{z}{\To_{\bar\uo} c^k - x^\mpi } =
 \sinner{s}{\To_{\bar\uo} c^k - \To_{\bar\uo}\To_\uo^*x^\mpi } =
 \sinner{\To^*_{\bar\uo} z}{c^k - \To_\uo^*x^\mpi } \to 0 \,.
\end{equation*}
Thus  $x^k = \To_{\bar\uo} c^k \rightharpoonup x^\mpi$, which  concludes the proof.
\end{proof}

\subsection{Numerical rates}\label{sec:num}
In this section, we present a brief example of numerical rates by comparing the filters of Example \ref{ex:caseA} and \ref{ex:caseB} with the soft thresholding filter. The underlying operator is the 2D Radon transform, and we utilize the wavelet vaguelette decomposition (DFD) with the Haar wavelet \cite{donoho1995nonlinear}. Figure \ref{fig:rates} plots, on the left, the filter function with $\kappa_\lambda=1/4$ and $\alpha=1/10$ for the regularizing filters in Example \ref{ex:caseA} and \ref{ex:caseB}, as well as the soft thresholding filter. On the right, the $\ell^2$-error of the reconstruction compared to the ground truth is plotted with respect to various percentages of added noise. A linear parameter choice, $\alpha = C \delta$ for some constant $C > 0$, is employed.

\begin{figure}[htb!]\label{fig:rates}
\centering
\begin{tikzpicture}[x=2cm, y=2cm, font=\small]
	\draw (0,0) node {\includegraphics [scale=0.56]{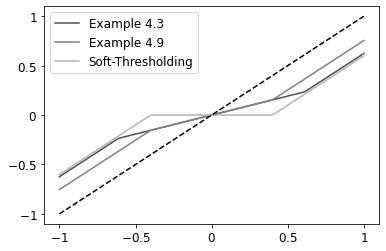}};
	\draw (4,0) node {\includegraphics [scale=0.56]{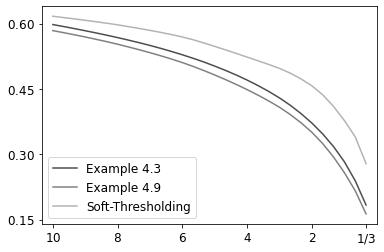}};
\end{tikzpicture}
\caption{On the left, filter functions for Example \ref{ex:caseA} and \ref{ex:caseB} (with $\kappa_\lambda=1/4$ and $\alpha=1/10$), as well as the soft thresholding filter, are displayed. On the right, the plot illustrates the $\ell^2$-error of reconstructions under varying percentages of added noise. A linear parameter choice is applied.}
\end{figure}

\section{Connection to PnP regularization}
\label{sec:pnp}

PnP regularization and variational regularization are closely linked. Moreover, PnP regularization operates with an implicitly provided regularization parameter, similar to the direct analysis. In this section, our goal is to clarify the connections between the PnP framework and non-linear filtered DFD. We introduce assumptions about the regularizing filter that enable us to align our method with PnP regularization, similar to variational regularization, thereby achieving strong stability under slightly stronger conditions. Furthermore, we employ direct analysis strategies to establish convergence results for a diagonalized version of PnP regularization under more relaxed assumptions, as needed in previous works.

The concept in PnP regularization is to find a fixed point of the operator
\begin{equation}
	\To_{\al,y} \coloneqq \Do_\al \circ \bigl(\id- \ga \nabla \norm{\Ao (\cdot) - y}^2/2) \bigr)  \colon \X \to \X \,,
\end{equation}
where $(\Do_\al)_{\al > 0}$ is a suitable family of regularization operators (or denoisers).
The PnP framework has been applied successfully  various fields  including image restoration \cite{Ca16, zhang2021plug} and inverse imaging  problems \cite{wei2020tuning, Wo19}.

If $\Do_\al$ is a proximity operator of a proper, convex and lower semi-continuous functional $\ga\Rc_\al$, then, by Lemma~\ref{lem:prox}~\ref{lem:prox3} we have $\fix(\To_{\al,y})= \argmin_{x \in \X} \norm{\Ao x-y}^2/2 + \Rc_\al$.

\begin{definition}[Family  of denoisers] \label{def:denoisers}
We call $(\Do_\al)_{\al >0}$   admissible  family  of denoisers if the following hold:
\begin{enumerate} [label=(D\arabic*), leftmargin =2.5em]
\item \label{def:D1} $\forall \al> 0 \colon \Lip(\Do_\al) < 1$.

\item\label{def:D2} $ \forall x \in \E  \coloneqq \bigcup_{\al >0} \ran(\Do_\al) \colon \Do_\al(x) \rightarrow x$ as $\al \to 0$.

\item  \label{def:D3}  $\forall B \subseteq \E$ bounded $ \colon \forall z \in \X \colon \sup_{x \in B} \inner{\Do_\al(x) -x}{z} \to 0$.
\item \label{def:D4}  $\forall x \in \E \colon \exists M_x < \infty \colon \forall \al >0 \colon  \norm{x-\Do_\al(x)} /(1- \Lip(\Do_\al)) \leq M_x$.
\end{enumerate}
\end{definition}

In \cite{ebner2022plug}, the following results for PnP as a regularization method, notably the first of its kind, have been derived.

\begin{lemma}[PnP regularization] \label{lem:PnP}
Suppose $(\Do_\al)_{\al>0}$ is an admissible  family  of denoisers, let  $ \ga \in (0, 2/\norm{\Ao}^2)$ and suppose $y, y^k \in \Y$ with $y^k \to y$ as $k \to \infty$. Then we have:
\begin{itemize}
\item Existence: $\To_{\al,y}$ has exactly one fixed point.
\item Stability: $\fix(\To_{\al,y^k}) \to \fix(\To_{\al,y})$  as $k \to \infty$.
\item Convergence: Assume $y = \Ao x$ for $x \in \E$ and suppose $\snorm{y^k-y} \leq \delta_k$  where $\delta_k \to 0$.
Consider $\al_k \to 0$ such that $(1-\Lip(\Do_{\al_k}))/\Lip(\Do_{\al_k}) \gtrsim \delta_k$ and take $x^k = \fix( \To_{\al_k,y^k})$. There exists a subsequence $(x^{k(\ell)})_{l \in \N}$ and a solution $x^\mpi$ of $\Ao x = y$ such that  $x^{k(\ell)} \rightharpoonup x^\mpi$.
If the solution is unique, then $x^k \rightharpoonup x^\mpi$.
\end{itemize}
\end{lemma}

\subsection{Reducing non-linear filtered DFD to PnP}

Building upon Lemma~\ref{lem:main}, we apply Proposition~\ref{lem:PnP} with $\Do_\al = \prox_{\ga \Rc_\al}$ and $\Ao = \Mo_\kao$. To ensure the admissibility of the family of denoisers $(\Do_\al)_{\al>0}$, we impose assumptions on the non-linear regularizing filter, contingent on the quasi-singular values and, consequently, on the operator.
One crucial condition for ensuring the admissibility of denoisers is contractiveness. Enforcing $\prox{\ga \Rc_\al}$ to be contractive means that the regularizing functional $\Rc_\al$ has to be strongly convex, and in turn, all $\qone_{\al,\la}$ have to be strongly convex.

\begin{assumption}[Reduction to PnP regularization]\label{ass:C}
Let $(\ph_{\al})_{\al > 0}$ be a non-linear regularizing  filter and  $\ga \in (0, 1/\max_\la \ka^2_\la)$ such that   for all $\al > 0$  there exist  $\ell_{\al} \in (0,1)$ with
\begin{enumerate} [label=(C\arabic*), leftmargin =2.5em]
\item\label{ass:C1} $\forall \la \in \La \colon \Lip(\ph_{\al}(\ka_\la,\cdot)) \leq \ga \ka_\la^2 \ell_{\al}  / ( 1-\ell_{\al}(1-\ga \ka_\la^2) )$.
\item\label{ass:C2} $\forall \la \in \La \colon  \exists C \in [1, \infty) \colon 1/(1-\ell_\al) \geq C\, \Rightarrow \, \abs{\ph_{\al}(\ka_\la,x)} \geq \frac{\ga \ka_\la^2 (1- C (1- \ell_\al))}{1 - (1- C (1- \ell_\al))(1- \ga \ka_\la^2)} \sabs{x}$.
\end{enumerate}
\end{assumption}

Next we give an example  of a family $(\ph_\al)_{\al > 0}$ that satisfies Assumption~\ref{ass:C} and  does not satisfy Assumption~\ref{ass:A}.

\begin{example}\label{ex:caseC}
Set $\ell_\al = 1 / (1+\al)$ and consider the function
\[
\ph_\al(\ka,x)= \begin{cases} \frac{\ga \ka^2 \ell_{\al}}{ 1-\ell_{\al}(1-\ga \ka^2) } \, x & \sabs{x} \leq \al / 3 \\
\frac{\ga \ka^2 \ell_{\al}}{ 1-\ell_{\al}(1-\ga \ka^2) } \, \sign(x) \al / 3  & \al/3 \leq \sabs{x} \leq 2 \al/3  \\
\frac{\ga \ka^2 \ell_{\al}}{ 1-\ell_{\al}(1-\ga \ka^2) } \, (x - \sign(x) \al / 3 ) & \sabs{x} \geq  2 \al /3 \,.
\end{cases}
\]
Then $(\ph_\al)_{\al > 0}$ is a non-linear regularizing filter and satisfies Assumption \ref{ass:C}.
Since $(\abs{\ph_{\al}(\ka,x)})_{\al > 0}$ is not monotone for some $x \in \R$, $(\ph_\al)_{\al > 0}$ does not satisfy Assumption~\ref{ass:A}.

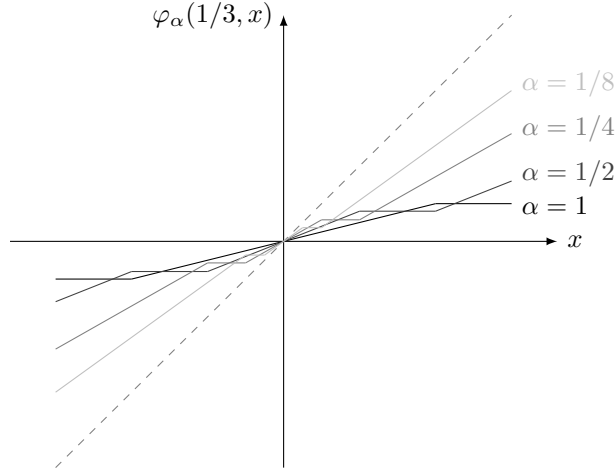
\begin{figure}[htb!]
\centering
\begin{tikzpicture}[>=latex, x=6cm, y=6cm, font=\small]
\draw[->] (-0.6, 0) -- (0.6, 0) node[right] {$x$};
\draw[->] (0, -0.5) -- (0, 0.5) node[left] {$\ph_\al(1/3,x)$};
\draw[dashed, gray] (-0.5, -0.5) -- (0.5, 0.5);
%\al=1 \ka=1/3
\draw[scale=1, domain=-(1/3*1):(1/3*1), smooth, variable=\x, black] plot ({\x}, {1/4*\x});
\draw[scale=1, domain=(1/3*1):0.5, smooth, variable=\x, black] plot ({\x}, {1/4*1/3});
\draw[scale=1, domain=-0.5:-(1/3*1), smooth, variable=\x, black] plot ({\x}, {1/4*-1/3});
\node at (0.5,0.08) [right, black] {$\al=1$};
%\al=1/2 \ka=1/3 \ga=3
\draw[scale=1, domain=-(1/3*1/2):(1/3*1/2), smooth, variable=\x, darkgray] plot ({\x}, {2/5*\x});
\draw[scale=1, domain=(1/3*1/2):(2/3*1/2), smooth, variable=\x, darkgray] plot ({\x}, {2/5*1/6});
\draw[scale=1, domain=-(2/3*1/2):-(1/3*1/2), smooth, variable=\x, darkgray] plot ({\x}, {2/5*-1/6});
\draw[scale=1, domain=(2/3*1/2):0.5, smooth, variable=\x, darkgray] plot ({\x}, {2/5*(\x-1/6)});
\draw[scale=1, domain=-(2/3*1/2):-0.5, smooth, variable=\x, darkgray] plot ({\x}, {2/5*(\x+1/6)});
\node at (0.5,0.15) [right, darkgray] {$\al=1/2$};
%\al=1/4 \ka=1/3
\draw[scale=1, domain=-(1/3*1/4):(1/3*1/4), smooth, variable=\x, gray] plot ({\x}, {4/7*\x});
\draw[scale=1, domain=(1/3*1/4):(2/3*1/4), smooth, variable=\x, gray] plot ({\x}, {4/7*1/12});
\draw[scale=1, domain=-(2/3*1/4):-(1/3*1/4), smooth, variable=\x, gray] plot ({\x}, {4/7*-1/12});
\draw[scale=1, domain=(2/3*1/4):0.5, smooth, variable=\x, gray] plot ({\x}, {4/7*(\x-1/12)});
\draw[scale=1, domain=-0.5:-(2/3*1/4), smooth, variable=\x, gray] plot ({\x}, {4/7*(\x+1/12)});
\node at (0.5,0.25) [right, gray] {$\al=1/4$};
%\al=1/8 \ka=1/3
\draw[scale=1, domain=-(1/3*1/8):(1/3*1/8), smooth, variable=\x, lightgray] plot ({\x}, {8/11*\x});
\draw[scale=1, domain=(1/3*1/8):(2/3*1/8), smooth, variable=\x, lightgray] plot ({\x}, {8/11*1/24});
\draw[scale=1, domain=-(2/3*1/8):-(1/3*1/8), smooth, variable=\x, lightgray] plot ({\x}, {8/11*-1/24});
\draw[scale=1, domain=(2/3*1/8):0.5, smooth, variable=\x, lightgray] plot ({\x}, {8/11*(\x-1/24)});
\draw[scale=1, domain=-0.5:-(2/3*1/8), smooth, variable=\x, lightgray] plot ({\x}, {8/11*(\x+1/24)});
\node at (0.5,0.35) [right, lightgray] {$\al=1/8$};
\end{tikzpicture}
\caption{Illustration of the filter  $(\ph_\al)_{\al > 0}$ from Example~\ref{ex:caseC} that  satisfies Assumption~\ref{ass:C} and does not satisfy Assumption~\ref{ass:A}.}
\end{figure}

\end{example}

Now, we can successfully simplify the non-linear filtered DFD to PnP regularization.

\begin{lemma}[Reduction to PnP regularization] \label{lem:PnP-reduce}
Let $(\ph_{\al})_{\al > 0}$ be non-linear regularizing filter satisfying Assumption \ref{ass:C}.
Then, $(\prox_{\ga \Rc_\al})_{\al > 0}$  is an admissible  family  of denoisers  in the sense of Definition \ref{def:denoisers}.
\end{lemma}

\begin{proof}
Let $y \in \R$. For $\qone \in \Ga_0(\R)$ and $\ga \ka^2<1$ we have
\begin{multline*}
\prox_{\qone}^{-1}(y) = y + \frac{1}{\ga}\partial \ga \qone(y) = y + \frac{1}{\ga\ka}\partial(\ga \qone(\ka (\cdot))) \left(\frac{y}{\ka}\right) \\
= y + \frac{1}{\ga\ka}(\prox_{\ga \qone(\ka\cdot)}^{-1} - \id ) \left(\frac{y}{\ka}\right)  = y - \Bigl(\frac{y}{\ga \ka^2}\Bigr) + \frac{1}{\ga\ka} \prox_{\ga \qone(\ka (\cdot))}^{-1} \left(\frac{y}{\ka}\right)
\end{multline*}
and therefore $\prox_{\ga \qone(\ka (\cdot))}^{-1} (y) - y = \ga \ka  ( \prox_{\qone}^{-1}(\ka y)- \ka y )$.
Now let $x \in \R$ be arbitrary.
If $\abs{\prox_{\qone}(x)} \geq (\ga \ka^2 t)/(1 + t(1- \ga \ka^2))) \sabs{x}$  we derive  $\abs{z - \ka y} \leq ((1-t)/(t \ga \ka)) \sabs{x}$ for all $z \in \prox_{\qone}^{-1}(\ka y)$.
Similar to the proof of Lemma \ref{prox-scaled} it follows that
\begin{equation} \label{eq:est1}
\bigl\lvert x-\prox_{\ga \qone(\ka (\cdot))}(x) \bigr\rvert \leq (1-t) \sabs{x} \,.
\end{equation}
Also following the proof of Lemma \ref{prox-scaled} one can show that for all $t \in [0,1)$,
\begin{align*}
&\forall x,y \in \R    \colon \abs{\prox_\qone(x)-\prox_\qone(y)} \leq \frac{\ga \ka^2 t}{1- t(1-\ga \ka^2)} \sabs{x-y} \\
& \qquad
\Rightarrow
\forall x,y \in \R \colon \bigl\lvert \prox_{\ga \qone(\ka (\cdot))}(x)-\prox_{\ga \qone(\ka (\cdot))}(y) \bigr\rvert  \leq t \sabs{x-y} \,.
\end{align*}

Let $c = (c_\la)_\la , d=(d_\la)_\la \in \ell^2(\La)$ be arbitrary and $\al > 0$ be fixed.
Then,
\begin{align*}
\snorm{\prox_{\ga \Rc_\al}(c) - \prox_{\ga \Rc_\al}(d)}^2 &= \sum_{\la \in \La} \bigl\lvert \prox_{\ga \qone_{\al,\la}(\ka_\la (\cdot))}(c_\la)-\prox_{\ga \qone_{\al,\la}(\ka_\la (\cdot))}(d_\la) \bigr\rvert^2  \\
&\leq \sum_{\la \in \La} \ell_{\al}^2 \abs{c_\la-d_\la}^2 = \ell_\al^2 \norm{c-d}^2.
\end{align*}
Since $\ell_\al^2 < 1$, the operator $\prox_{\ga \Rc_\al}$ is a contraction for all $\al> 0$, which is \ref{def:D1}.
To show \ref{def:D2}, let  $x =(x_\la)_\la \in \ell^2(\La)$.
By assumption, for all $y \in \R$ and for all $ \la \in \La$, we have $\lim_{\al \to 0} \ph_{\al}(\ka_\la,y) = y$ and by Lemma \ref{lem:inv-conv}, we also have $\lim_{\al \to 0} \ph_{\al}(\ka_\la,\cdot)^{-1}(y) = \{y\}$.
It follows that
\begin{equation*} \prox^{-1}_{\ga \qone_{\al,\la}(\ka_\la (\cdot) )}(y) = \ka_\la \ph_{\al}(\ka_\la,\cdot)^{-1}(\ka_\la y ) - \ka_\la^2 y + y  \rightarrow \{\ka_\la^2 y - \ka_\la^2 y + y\} = \{y\} \quad \text{ as } \al \to 0 \,.
\end{equation*}
Since $\prox_{\ga \qone_{\al,\la}(\ka_\la (\cdot) )}$ is increasing and nonexpansive with $\prox_{\ga \qone_{\al,\la}(\ka_\la (\cdot) )}(0)=0$, according to Lemma~\ref{lem:inv-conv}, it holds $\lim_{\al \to 0} \prox_{\ga \qone_{\al,\la}(\ka_\la (\cdot) )}(y) = y$ for all $y \in \R$.
Now let $\epsilon > 0$. Then there exists a finite subset $\Omega \subseteq \La$ such that $\sum_{\la \in \La\setminus\Omega} \sabs{x_\la}^2 < \epsilon /2$ and there exist $\tilde{\al} >0$ small enough such that for all $\al< \tilde{\al}$, we have
\[\sum_{\la \in \Omega} \bigl\lvert \prox_{\ga \qone_{\al,\la}(\ka_\la (\cdot) )}(x_\la)-x_\la \bigr\rvert^2 < \frac{\epsilon}{2}. \]
From Assumption~\ref{def:filter} it follows that $\sabs{\prox_{\ga \qone_{\al,\la}(\ka_\la (\cdot) )}(y)-y} \leq \abs{y}$, hence we have
\begin{multline*}
\snorm{\prox_{\ga \Rc_\al}(x)-x}^2 = \sum_{\la \in \Omega} \bigl\lvert\prox_{\ga \qone_{\al,\la}(\ka_\la (\cdot) )}(x_\la)-x_\la\bigl\lvert^2 + \sum_{\la \in \La\setminus\Omega} \bigl\lvert\prox_{\ga \qone_{\al,\la}(\ka_\la (\cdot) )}(x_\la)-x_\la \bigr\rvert^2 \\
 \leq \sum_{\la \in \Omega} \bigl\lvert \prox_{\ga \qone_{\al,\la}(\ka_\la (\cdot) )}(x_\la)-x_\la \bigr\rvert^2 + \sum_{\la \in \La\setminus\Omega} \sabs{x_\la}^2 < \frac{\epsilon}{2} + \frac{\epsilon}{2} = \epsilon \,.
\end{multline*}
Hence $\lim_{\al \to 0} \prox_{\ga \Rc_\al}(x) = x$ for all $x \in \ell^2(\La)$.
Since $\prox_{\ga \Rc_\al} \rightarrow \id$, we have
$1= \lim_{\al \to 0} \Lip(\prox_{\Rc_\al}) \leq \lim_{\al \to 0} \ell_\al \leq 1$ which is \ref{def:D2}. Let $z \in \ell^2(\La)$ and $B$ be a bounded set.
By \eqref{eq:est1} for $\al$ small enough we have
\begin{multline*}
\sabs{\sinner{\prox_{\ga \Rc_\al}(x)-x}{z}} \leq \sum_{\la \in \La} \sabs{\prox_{\ga \qone_{\al,\la}(\ka_\la (\cdot) )}(x_\la) - x_\la} \abs{ z_\la} \\
\leq \sum_{\la \in \La}  C (1-\ell_\al) \sabs{x_\la} \abs{z_\la} \leq C (1- \ell_\al) \norm{x} \snorm{z} \to 0
\end{multline*}
since $\ell_\al \rightarrow 1$ as $\al \to 0$. This shows \ref{def:D3}.
By \eqref{eq:est1} for all $x \in \ell^2(\La)$ and $\al$ small enough we have
\begin{multline*}
\frac{\norm{\prox_{\ga \Rc_\al}(x)-x}^2}{(1- \Lip(\prox_{\ga \Rc_\al}))^2}
 =\frac{\sum_{\la \in  \La} \sabs{\prox_{\ga \qone_{\al,\la}(\ka_\la (\cdot) )}(x_\la) - x_\la}^2}{(1- \ell_{\al})^2} \\
 \leq \frac{\sum_{\la \in \La} C^2 (1- \ell_\al)^2 \sabs{x_\la}^2}{(1- \ell_{\al})^2} = C^2 \norm{x}^2 \,
\end{multline*}
which shows \ref{def:D4}.
\end{proof}

\begin{proposition} \label{prop:caseB}
Let $(\ph_\al)_{\al > 0}$ be a non-linear regularizing filter such that Assumption \ref{ass:C} is satisfied.
Suppose $\al,\al^k > 0$, and $z, z^k \in \ell^2(\La)$ for $k \in \N$ with $(z^k)_{k \in \N} \to z$.
\begin{itemize}
\item Existence: $\dom(\Mo_{\kao}^\mpi \circ \Phi_{\al, \kao}) = \ell^2(\La)$.
\item Stability: $\Mo_{\kao}^\mpi \circ \Phi_{\al, \kao} (z^k) \to \Mo_{\kao}^\mpi \circ \Phi_{\al, \kao} (z)$ as $k \to \infty$.
\item Convergence: Assume $z =\Mo_{\kao}c^\mpi$ with $z \in \dom(\Phi_{\al, \kao}^{-1})$ for all $\al> 0$ and suppose $\snorm{z^k -z} \leq \delta_k$. Let  $\delta_k, \al_k \to 0$ and $(1-\ell_{\al_k})/\ell_{\al_k} \gtrsim \delta_k$. Then $\Mo_{\kao}^\mpi \circ \Phi_{\al_k, \kao} (z^k)  \rightharpoonup c^\mpi$ as $k \to \infty$.
\end{itemize}
\end{proposition}

\begin{proof}
Follows from Lemmas \ref{lem:PnP-reduce} and  \ref{lem:PnP} and the injectivity of $\Mo_{\kao}$.
\end{proof}

\begin{remark}[Same parameter choice]
Opting for $\ell_\al = 1/(1+\sqrt{\al})$, the conventional parameter choice $\delta_k^2/\al_k \rightarrow 0$ suffices. 
\end{remark}

\subsection{PnP denoiser as a diagonal operator}

Since we can express the filtered DFD in the form of plug-and-play regularization, and convergence is demonstrated under weak assumptions in the direct analysis of Section \ref{sec:conv}, this extends to a version of PnP that goes beyond the scope of \cite{ebner2022plug}.

After diagonalizing the operator $\Ao = \To_{\bar \vo} \Mo_{\kao} \To_{\uo}^*$ with a frame decomposition, we apply PnP to regularize the discontinuous diagonal operator $\Mo_{\kao}$. This involves considering the fixed points of $\Do_\al \circ \left(\id - \ga \Mo_{\kao} (\Mo_\kao (\cdot) - z) \right) \colon \ell^2(\La) \to \ell^2(\La)$. Assuming that $\Do_\al$ is also diagonal, meaning it consists of functions $d_{\al,\la} \colon \R \to \R$, then $\Do_\al((c_\la)_{\la \in \La}) = (\doo_{\al,\la}(c_\la) )_{\la \in \La}$.

Now let $\doo_{\al,\la}$ satisfy the following conditions:

\begin{enumerate}
\item $\forall \al, \la \colon \doo_{\al,\la}$ is monotonically increasing, nonexpansive, and $d_{\al,\la}(0)=0$.
\item $\forall x, \la \colon  (\abs{\doo_{\al,\la}(x)})_{\al > 0}$ is monotonically increasing for $\al\downarrow 0$ and converges to $x$.
\item $ \exists d, e> 0$ $\forall x, \al, \la \colon \sabs{x} \leq d \al / \ka_\la^2 \Rightarrow \abs{\doo_{\al,\la}(x)} \leq\left(1+\ga \ka_\la^2\left(\frac{e \sqrt{\al}}{\ka_\la}-1\right)\right)^{-1} \sabs{x}$.
\end{enumerate}
Then, as in Lemma \ref{lem:main}\ref{lem:main1}, there  exist $r_{\al,\la} \in \Gamma_0(\R)$ with $r_{\al,\la} \geq r_{\al,\la}(0)=0$ such that $\doo_{\al,\la} = \prox_{r_{\al,\la}}$.
One verifies that
\begin{align*}
\fix(\doo_{\al,\la} \circ (\id - \ga \ka_\la(\ka_\la (\cdot) - z_\la))) &= \argmin_{x \in \R} \abs{\ka_\la x - z_\la}^2/2 + \ga^{-1} r_{\al,\la}(x)\\
&= \ka_\la^{-1} \argmin_{x \in \R} \abs{x - z_\la}^2/2 + \ga^{-1} r_{\al,\la}(x/\ka_\la)\\
&= \ka_\la^{-1} \prox_{\ga^{-1} r_{\al,\la}(\ka_\la^{-1}\cdot)}.
\end{align*}
Hence, this PnP regularization reduces to a non-linear filtered DFD, where $\ph_{\al}(\ka_\la, \cdot) = \prox_{\ga^{-1} r_{\al,\la}(\ka_\la^{-1}\cdot)}$.
Condition (b) for $\doo_{\al,\la}$ directly implies assumption \ref{ass:A1} in Section \ref{sec:conv}, and condition (c) allows us to demonstrate that
\[
\sabs{x} \leq d \al / \ka_\la \Rightarrow \abs{\prox_{\ga^{-1} r_{\al,\la}(\ka_\la^{-1}\cdot)(x)}} \leq\frac{e\sqrt{\al}}{\ka_\la} \sabs{x}.
\]
Now, we obtain the following convergence result.
Let $(\delta_k)_k, (\al_k)_k$  null sequences with $\delta_k^2/\al_k \to 0$ and $y \in \ran(\Ao)$.
Suppose $\To_{\vo}^* y \in \ran(\Do_{\tilde \al})$ for some $\tilde{\al}> 0$ and let $(y^k)_k \in \Y^\N$ with $\snorm{y^k -y} \leq \delta_k$.
Define $c^k \coloneqq \fix ( \Do_{\al_k} \circ ( \id - \ga \nabla  \snorm{\Mo_{\kao} (\cdot) - \To_\vo^*(y^k)}^2/2 )$, then $ \To_{\bar \uo} c^k \rightharpoonup \Ao^\mpi y$ as $k \to \infty$.

\section{Conclusion}

\label{sec:conclusion}

In this paper, we have analyzed non-linear diagonal frame filtering for the regularization of inverse problems. Compared to the previously analyzed linear filters, non-linear filters can better exploit the specific structure of the target signal and the noise. We introduced assumptions on the non-linear regularizing filters to prove stability and weak convergence in the context of regularization theory, and established connections to variational regularization and PnP regularization. Future research might focus on convergence in terms of norm topology and the derivation of convergence rates.

\appendix

\section{Remaining proofs}

\label{app:proofs}

In this  appendix  we provide  proofs of technical lemmas  stated in Section \ref{sec:lemmas}.

\subsection{Proof of Lemma \ref{lem:inv-conv}}

\label{app:proof1}

Since $\ph_\al$ is increasing and $\ph_\al(0)=0$, the pre-image is either empty or a closed interval of the form $\ph_\al^{-1}(x)=[p_\al,q_\al]$ depending on $\al$ where we have three cases depending on $x$.
If $x \geq 0$ then $0 \leq p_\al \leq q_\al < \infty$, if $x \leq 0$ then $-\infty < p_\al \leq q_\al \leq 0$ and if $x=0$ then $\infty < p_\al \leq 0 \leq q_\al < \infty$.
Furthermore, since $\ph_\al$ is nonexpansive we have $\abs{p_\al}, \abs{q_\al} \geq \sabs{x}$.

Suppose \ref{lem:inv-conv1} holds.
By assumption we have $0 \in \ph_{\al}^{-1}(0)$.
Let $x \neq 0$ and suppose there is a null sequence $\al_k$ such that $\forall k \in \N \colon \ph_{\al_k}^{-1}(x) = \emptyset$ .
Since $\ph_{\al_k}$ is continuous, we have $\abs{\ph_{\al_k}(2x)} < \sabs{x}$ otherwise it contradicts the intermediate value theorem and therefore $\abs{\ph_{\al_k}(2x)-2x} > \sabs{x}$ for all $k \in \N$, which is a contradiction. Thus, for all $x \in \R$ there exists $\tilde{\al}> 0$ such that $\forall \al \in (0,\tilde{\al}) \colon \ph_\al^{-1}(x) \neq \emptyset$.  Now fix any $x > 0$. Then for all $\al < \tilde{\al}$
\[\sup_{y \in \ph_\al^{-1}(x)} \abs{y-x} = \sup_{y \in [p_\al,q_\al]} \abs{y-x} = \abs{q_\al-x} \]
Suppose that $q_\al \nrightarrow x$, then $\limsup_{\al \to 0} q_\al \eqqcolon q > x$ (and $q < \infty$, otherwise there would exist a null sequence $\al_k$ such that $\lim_{k \to \infty} \ph_{\al_k}(y) = y \leq x$ for all $y \in \R$, which is a contradiction).
Then there exists a null sequence $\al_k$ such that $ \abs{\ph_{\al_k}(q_{\al_k})-\ph_{\al_k}(q)} \leq \abs{q_{\al_k} - q} \to 0$
as $k \to \infty$. But $\abs{\ph_{\al_k}(q_{\al_k})-\ph_{\al_k}(q)} = \abs{x-\ph_{\al_k}(q)} \rightarrow \abs{x-q} > 0$, as $k \to \infty$, which is a contradiction.

To show the converse implication, suppose  that \ref{lem:inv-conv2} holds and let $x \in \R$ be fixed.
With  $z \coloneqq \ph_\al(x)$, we have $x \in \ph_\al^{-1}(z)$  and
$ \abs{\ph_\al(x)-x} = \abs{z-x} \leq \sup_{x \in \ph_\al^{-1}(z)}\abs{z-x} \to 0$.

\subsection{Proof of Lemma \ref{lem:r-estimate}}

 \label{app:proof2}

Fix $b,c > 0$ such that \eqref{eq:ass-restimate}  is satisfied.
We start by considering the case where $y \geq 0$ with $y \leq c \ka$.
Since $\qone$ is nonnegative and $\qone(0)=0$, we have $\prox_{\al \qone}(0)=0$ and $\prox_{\al \qone}$ is positive for positive arguments  and negative for negative arguments.
For  $y \in [0,c\ka]$ we have $\prox_{\al s}^{-1}(y) \leq \min(\prox_{\al s}^{-1}(c \ka))$ and hence from \eqref{eq:ass-restimate} it follows $
\prox^{-1}_{\al \qone}(y) \geq y (\ka^2+\al b)/\ka^2$.
Note that $\prox^{-1}_{\al \qone}(y)$ can be set-valued in which case these inequalities are understood  to be  satisfied for all elements.
By Lemma~\ref{lem:prox}\ref{lem:prox1} we get $(\id + \al \partial \qone)(y) \geq y + y \al b/\ka^2$ which  is equivalent to  $\partial \qone(y) \geq yb /\ka^2 = \partial (y^2 b / (2 \ka^2))$. With  $\qone(0)=0 $ we get
$\qone(y) \geq (b/2) \cdot \sabs{y/\ka}^2$ for all $y \in [0, c \ka]$.
Now consider the case $y > c\ka$.
Since $\qone$ is convex, the sub-gradients increase.
At point $\ka c $ the sub-gradients of $\qone$ are grater than $bc /\ka$, hence for all $y \geq c\ka$ we have $\qone(y) \geq y \, bc / \ka + a $ and by taking  $y= \ka c$ we get $a \geq - bc^2/2$.  Similar arguments for $y<0$ complete the proof.

\subsection{Proof of Lemma \ref{prox-scaled}}

\label{app:proof3}

Let $y \in \R$ be arbitrary.
We have
\begin{multline*}
\prox_{\qone}^{-1}(y)
= y + \frac{1}{\ga}\partial \ga \qone(y)
= y + \frac{1}{\ga\ka}\partial(\ga \qone(\ka (\cdot)))  (y/\ka ) \\
= y + \frac{1}{\ga\ka}(\prox_{\ga \qone(\ka(\cdot))}^{-1} - \id ) ( y/\ka )
= y - y/(\ga \ka^2 ) + \frac{1}{\ga\ka} \prox_{\ga \qone(\ka (\cdot))}^{-1} (y/\ka).
\end{multline*}
Now let $\sabs{x} \leq \ga \al$ and define $a \coloneqq \ka \prox_{\ga\qone(\ka (\cdot))}(x)$. Then,
\begin{equation*}
b \coloneqq \ka \bigl(1-  1/(\ga\ka^2)\bigr) \prox_{\ga \qone(\ka (\cdot))}(x) + x/(\ga \ka) \in \prox_\qone^{-1}(a)
\end{equation*}
and $\abs{b} = (\ga \ka)^{-1} \sabs{x-(1-\ka^2\ga)\prox_{\ga \qone(\ka (\cdot))}(x)} \leq \sabs{x}/(\ga \ka) \leq \al/\ka$.
Thus, we have
\begin{multline*}
\bigl\lvert \prox_{\ga \qone(\ka (\cdot))}(x) \bigr\rvert   = \Bigl\lvert \frac{a}{\ka} \Bigr\rvert = \frac{1}{\ka} \abs{\prox_\qone(b)} \leq \frac{\ga \ka t}{1-t(1-\ga \ka^2)} \abs{b} \\
\leq \frac{t}{1-t(1-\ga \ka^2)} \bigl(\sabs{x} - (1-\ga \ka^2) \lvert\prox_{\ga \qone(\ka (\cdot))}(x) \rvert \bigr) \,,
\end{multline*}
from which it follows that $t^{-1} \sabs{\prox_{\ga \qone(\ka (\cdot))}(x)}  \leq \sabs{x}$.

\end{document}